\documentclass[11pt]{article}

\usepackage{amsmath}
\usepackage{amsthm}
\usepackage{amssymb}
\usepackage{latexsym}
\usepackage{pstricks}
\usepackage{mathpazo}
\usepackage{graphicx, pst-plot, pst-node, pst-text, pst-tree}
\usepackage{titlefoot}
\usepackage{titlesec}
\usepackage[small,it]{caption}
\usepackage{color}
\definecolor{lightgray}{rgb}{0.8, 0.8, 0.8}
\definecolor{darkgray}{rgb}{0.7, 0.7, 0.7}
\definecolor{darkblue}{rgb}{0, 0, .4}

\usepackage[bookmarks,breaklinks]{hyperref}
\hypersetup{
        colorlinks=true,
        linkcolor=darkblue,
        anchorcolor=darkblue,
        citecolor=darkblue,
        urlcolor=darkblue,
        pdfpagemode=UseThumbs,
        pdftitle={Simple extensions of combinatorial structures},
        pdfsubject={Combinatorics},
        pdfauthor={Robert Brignall, Nik Ruskuc, and Vincent Vatter},
}

\newtheorem{theorem}{Theorem}[section]
\newtheorem{proposition}[theorem]{Proposition}
\newtheorem{lemma}[theorem]{Lemma}

\newtheorem{example}[theorem]{Example}

\newtheorem*{claim}{Claim}

\newcounter{todocounter}

\makeatletter
\newfont{\footsc}{cmcsc10 at 8truept}
\newfont{\footbf}{cmbx10 at 8truept}
\newfont{\footrm}{cmr10 at 10truept}
\renewenvironment{abstract}%
                {
                  \begin{list}{}%
                     {\setlength{\rightmargin}{1in}%
                      \setlength{\leftmargin}{1in}}%
                   \item[]\ignorespaces\begin{small}}%
                 {\end{small}\unskip\end{list}}

\datefoot{\today}
\amssubj{indecomposable graph, modular decomposition, prime graph, simple permutation, simple poset, simple tournament, substitution decomposition}

\newpagestyle{main}[\small]{
        \headrule
        \sethead[\usepage][][]
        {\sc Simple Extensions of Combinatorial Structures}{}{\usepage}}

\title{\sc{Simple Extensions of Combinatorial Structures}}
\author{%
Robert Brignall\footnote{Supported by the Heilbronn Institute for Mathematical Research}\\[-0.25ex]
\small Department of Mathematics\\[-0.5ex]
\small University of Bristol\\[-0.5ex]
\small Bristol,  BS8 1TW\\[-0.5ex]
\small England\\[1.5ex]
Nik Ru\v{s}kuc\\[-0.25ex]
\small Department of Mathematics and Statistics\\[-0.5ex]
\small University of St Andrews\\[-0.5ex]
\small St Andrews, KY16 9SS\\[-0.5ex]
\small Scotland\\[1.5ex]
Vincent Vatter\footnote{This research was conducted while V. Vatter was a member of the School of Mathematics and Statistics at the University of St Andrews, supported by EPSRC grant GR/S53503/01.}\\[-0.25ex]
\small Department of Mathematics\\[-0.5ex]
\small Dartmouth College\\[-0.5ex]
\small Hanover, New Hampshire 03755\\[-0.5ex]
\small U.S.A.\\[-10pt]
}

\titleformat{\section}
        {\large\sc}
        {\thesection.}{1em}{}

\date{}

\begin{document}
\maketitle

\newcommand{\Sub}{\operatorname{Sub}}
\newcommand{\Av}{\operatorname{Av}}
\newcommand{\rect}{\mathrm{rect}}
\newcommand{\A}{\mathcal{A}}
\newcommand{\B}{\mathcal{B}}
\newcommand{\C}{\mathcal{C}}
\newcommand{\D}{\mathcal{D}}
\renewcommand{\L}{\mathcal{L}}
\renewcommand{\P}{\mathcal{P}}
\newcommand{\Q}{\mathcal{Q}}
\newcommand{\R}{\mathcal{R}}
\renewcommand{\S}{\mathcal{S}}
\newcommand{\W}{\mathcal{W}}
\newcommand{\Si}{\operatorname{Si}}
\newcommand{\dom}{\operatorname{dom}}

\newcommand{\In}{\operatorname{In}}
\newcommand{\Out}{\operatorname{Out}}
\newcommand{\Core}{\operatorname{Core}}
\newcommand{\Int}{\operatorname{Int}}
\newcommand{\Ext}{\operatorname{Ext}}
\newcommand{\updown}{{\uparrow\hspace{-2pt}\downarrow}}
\newcommand{\downup}{{\downarrow\hspace{-2pt}\uparrow}}

\newcommand{\vspc}{\vspace{2mm}}

\newcommand{\OEIS}[1]{(sequence #1 of~\cite{sloane:the-on-line-enc:})}

\pagestyle{main}

\begin{abstract}
An interval in a combinatorial structure $R$ is a set $I$ of points which are related to every point in $R\setminus I$ in the same way.
A structure is simple if it has no proper intervals.
Every combinatorial structure can be expressed as an inflation of a simple structure by structures of smaller sizes --- this is called the substitution (or modular) decomposition.
In this paper we prove several results of the following type: An arbitrary structure $S$ of size $n$ belonging to a class $\mathcal{C}$ can be embedded into a simple structure from $\mathcal{C}$ by adding at most
$f(n)$ elements.
We prove such results when $\mathcal{C}$ is the class of all tournaments, graphs, permutations, posets, digraphs, oriented graphs and general relational structures containing a relation of arity greater than 2.  The function $f(n)$ in
these cases is $2$, $\lceil \log_2(n+1)\rceil$, $\lceil (n+1)/2\rceil$, $\lceil (n+1)/2\rceil$, $\lceil \log_4(n+1)\rceil$, $\lceil \log_3(n+1)\rceil$ and $1$, respectively.  In each case these bounds are best possible.
\end{abstract}

\section{Introduction}

Relational structures -- objects governed by a given set of relations over some ground set -- provide a general setting in which to study a wide variety of combinatorial structures including graphs, tournaments, posets and permutations. For each of these types of structure, the \emph{substitution decomposition} (also called the \emph{modular decomposition}) can be used to describe larger objects in terms of smaller ones. The elemental building blocks arising in this decomposition are the \emph{simple}\footnote{The terms indecomposable, prime, and primitive have been used in place of simple in other contexts.} structures.

In this paper, rather than decompose a combinatorial structure to describe it in terms of smaller simple ones, we will be embedding structures into larger simple ones. In particular, we consider the following question: given a class $\mathcal{C}$ of relational structures, what is the minimal function $f(n)$ so that any structure in $\mathcal{C}$ of size $n$ can be embedded in a simple structure from $\mathcal{C}$ containing at most $f(n)$ additional elements?

This question has its origins in the early 1970s with Erd\H{o}s {\itshape et al.}~\cite{erdos:some-remarks-on:}, who showed in the case of tournaments that $f(n)=2$ for all $n$. In a subsequent paper~\cite{erdos:simple-one-poin:} they also showed that the only tournaments which could not be embedded in a simple tournament with $1$ additional vertex were transitive tournaments with an odd number of vertices. The case of complete graphs was also answered around the same time in Sumner's thesis~\cite{sumner:indecomposable-:}, who showed that $f(n)=\lceil\log_2(n+1)\rceil$. In this paper, we answer the question completely for many standard combinatorial structures, in addition to more general relational structures containing relations of higher arity. These results rest on a unifying framework that uses the substitution decomposition, but each class of structures presents different problems and this is reflected in the variety of answers attained.

For the rest of this section, we will introduce important definitions and present the standard framework for our approach to the problem. In Section~\ref{sec-tournaments} we show how the case of tournaments fits into the framework, then in Sections~\ref{sec-graphs}---\ref{sec-digraphs} we demonstrate the cases of graphs, permutations, posets, digraphs and oriented graphs. Section~\ref{sec-higher-arity} discusses more general results for higher arity relational structures, and some concluding remarks are given in Section~\ref{sec-conclusion}.

A $k$-\emph{ary relation} $R$ on a set $A$ is a subset of $A^k$. Informally, a \emph{relational structure} is an ordered sequence of relations over some set $A$.
More specifically, we may define a \emph{relational language},
$\mathcal{L}$, to be a set of \emph{relational symbols}, and for each symbol $R$ there is a positive integer $n_R$ denoting the arity of $R$. A relational structure
$\mathcal{A}$ whose relational symbols are those of $\mathcal{L}$ is then
defined by its \emph{ground set} $A=\dom(\mathcal{A})$ and a set of subsets
$R^\mathcal{A}\subseteq A^{n_R}$ for each $R\in
\mathcal{L}$. Such a structure may also be called an
$\mathcal{L}$-\emph{structure}. When it does not cause any confusion, we will omit the superscript $\mathcal{A}$ from relations.

Most standard combinatorial structures can be viewed as relational structures;
they differ in the number of relations involved, their arities, and extra properties they are required to posses. This viewpoint offers a unifying framework for investigating different combinatorial structures `in parallel'.

\begin{example}
A graph $G$ can be viewed as a relational structure
with a single binary operation which is required to be symmetric and irreflexive.
Thus, in this case the language is  $\mathcal{L} = \{E, n_E=2\}$,
$\dom(\mathcal{A}) = V(G)$, and $(x,y) \in E$ if and only if $x\sim
y$ in $G$.
\end{example}

\begin{example}
Here are some further combinatorial structures viewed as relational structures:
\begin{itemize}
\item[(i)]
Oriented graphs: $\mathcal{L}=\{\rightarrow,n_\rightarrow=2\}$ where $\rightarrow$ is asymmetric,
i.e. $(x\rightarrow y) \Rightarrow (y\not\rightarrow x)$.
\item[(ii)]
Tournaments: $\mathcal{L}=\{\rightarrow,n_\rightarrow=2\}$, but where $\rightarrow$ is trichotomous:
for any $x,y$, precisely one of $x=y$, $x\rightarrow y$ or $y\rightarrow x$ is true.
\item[(iii)]
Digraphs: $\mathcal{L}=\{\rightarrow,n_\rightarrow=2\}$, a `generic' structure with a single binary relation.
\item[(iv)]
Posets: $\mathcal{L}=\{<,n_<=2\}$, where $<$ is asymmetric and transitive.
\item[(v)]
Linear orders:  $\mathcal{L}=\{<,n_<=2\}$, where $<$ is asymmetric, transitive and trichotomous.
\item[(vi)]
Permutations: $\mathcal{L} = \{<,\prec,n_<=2,n_\prec=2\}$, where each of $<$ and $\prec$ are linear orders.
For a permutation $\pi$ of $[n]$ these two orders are defined as follows:
$<$ is the normal ordering on $[n]$, while $i\prec j$ if and only if $\pi(i) < \pi(j)$.
\end{itemize}
\end{example}

A set $I\subseteq A=\dom(\mathcal{A})$ is an \emph{interval} if
for every $R\in\mathcal{L}$, all $x,y\in I$, and all ($n_R-1$)-tuples
$(x_1,\ldots,x_{i-1},x_{i+1},\ldots,x_{n_R})\in A^{n_R-1}\setminus I^{n_R-1}$ we have
\begin{displaymath}
(x_1,\ldots,x_{i-1},x,x_{i+1},\ldots,x_{n_R})\in R^\mathcal{A} \quad \Longleftrightarrow \quad
(x_1,\ldots,x_{i-1},y,x_{i+1},\ldots,x_{n_R})\in R^\mathcal{A}.
\end{displaymath}
Intuitively, in relationships involving elements outside an interval $I$, any two elements of $I$ are interchangeable.
Every singleton set $\{x\}\subseteq A$ is an interval, as is all of $A$ and the empty set.  All other interval are said to be \emph{proper}, and a structure is \emph{simple} if it has no proper intervals.

Let $\mathcal{A}$ and $\mathcal{B}$ be two structures over the same language $\mathcal{L}$.
We say that $\mathcal{A}$ is a \emph{substructure} of $\mathcal{B}$, or, equivalently, that
$\mathcal{B}$ is an \emph{extension} of $\mathcal{A}$, if $A\subseteq B$ and
$R^{\mathcal{A}}=R^{\mathcal{B}}\!\!\restriction_{A}$ for every $R\in\mathcal{L}$.
If, in addition, $\mathcal{B}$ is simple, we say that $\mathcal{B}$ is a \emph{simple extension} of $\mathcal{A}$.   Our aim is therefore to show that in a class of $\mathcal{L}$-structures $\mathcal{C}$, for every $\mathcal{A}\in\mathcal{C}$ there exists
a simple extension $\mathcal{B}\in\mathcal{C}$ of $\mathcal{A}$, and more importantly, to bound $|B\setminus A|$ as a function of $|A|$.

To establish embeddings into simple $\mathcal{L}$-structures we typically need to
look at the intervals of $\mathcal{A}$.
This is easier to establish by first defining the reverse
process. Given an $\mathcal{L}$-structure $\mathcal{S}$, we say that $\mathcal{A}$ is an \emph{inflation} of
$\mathcal{S}$ by the $\mathcal{L}$-structures
$\{\mathcal{A}_s : s\in S\}$, denoted
$\mathcal{A}=\mathcal{S}[\mathcal{A}_s:s\in S]$, if
$\mathcal{A}$ is obtained by replacing each element
$s\in S$ with a set of elements $A_s=\dom(\mathcal{A}_s)$
(assuming beforehand, of course, that all these domains are disjoint) that form an
interval in the $\mathcal{A}$.
More precisely, $A=\dom{\mathcal{A}}=\bigcup_{s\in S} A_s$, and for every $R\in\mathcal{L}$, and every $n_R$-tuple $(a_1,\ldots,a_{n_R})\in A$ where $a_i\in A_{s_i}$, we have
\begin{displaymath}
R^{\mathcal{A}}(a_1,\ldots,a_{n_R})\Leftrightarrow \left\{
\begin{array}{ll}
R^{\mathcal{A}_s}(a_1,\ldots,a_{n_R}) & \mbox{if } s_1=\ldots=s_{n_R}=s; \mbox{ or}\\
R^{\mathcal{S}}(s_1,\ldots, s_{n_R}) & \mbox{otherwise.}
\end{array} \right.
\end{displaymath}
If $\mathcal{A}$ is an inflation of $\mathcal{S}$, we also say that $\mathcal{S}$ is a
\emph{quotient} of $\mathcal{A}$.
We are particularly interested in the case where $\mathcal{S}$ is
simple --- Theorem \ref{substdecompthm} below gives the uniqueness of such an $\mathcal{S}$,
which will be called the \emph{simple quotient} of $\mathcal{A}$. This decomposition is the \emph{substitution} (or \emph{modular}) \emph{decomposition}.

The notion of modular decomposition dates back at least to a 1953 talk
of Fra\"{\i}ss\'e.  However, only the abstract of this
talk~\cite{fraisse:on-a-decomposit:} survives.  The first article
using modular decompositions seems to be
Gallai~\cite{gallai:transitiv-orien:} (for an English translation,
see~\cite{gallai:a-translation-o:}), who applied them to the study of
transitive orientations of graphs. Since then it has reappeared in settings ranging from game theory to combinatorial optimisation
-- see M\"ohring~\cite{mohring:an-algebraic-de:} or M\"ohring and Radermacher~\cite{mohring:substitution-de:} for extensive references.
Before stating and proving the general uniqueness theorem, we quote two basic facts about intervals, which follow from elementary considerations.  The first of these facts is proved explicitly by F\"{o}ldes \cite{foldes:on-intervals-in:}.

\begin{proposition}\label{prop-foldes}
For any two intervals $I$ and $J$ of the $\mathcal{L}$-structure $\mathcal{A}$:
\begin{enumerate}
\item[\textup{(a)}] If $I\cap J \neq \emptyset$ then $I\cap J$ and $I\cup J$ are intervals of $\mathcal{A}$.
\item[\textup{(b)}] $I\setminus J$ is an interval of $\mathcal{A}$.
\end{enumerate}
\end{proposition}

\begin{theorem}\label{substdecompthm}Let $\mathcal{A}$ be an $\mathcal{L}$-structure for some language $\mathcal{L}$. There exists a unique simple $\mathcal{L}$-structure $\mathcal{S}$ such that $\mathcal{A}=\mathcal{S}[\mathcal{A}_s:s\in S]$. Moreover, when $|S| > 2$,
the structures $\mathcal{A}_s$ are also uniquely determined.
\end{theorem}

\begin{proof}
Let $M$ denote the set of all intervals, except $A=\dom(\A)$, which are not contained in another proper interval.

If two intervals $I,J\in M$ intersect, then Proposition~\ref{prop-foldes} (a) shows that $I\cup J$ is also an interval, which, unless $I\cup J=A$, contradicts the definition of $M$.
If $I\cup J=A$, then Proposition~\ref{prop-foldes} (b) shows that $J\setminus I$ is an interval, so $\mathcal{A}$ can be written as the inflation
$\mathcal{S}[\mathcal{A}_1, \mathcal{A}_2]$
of a two-element structure $\mathcal{S}$ (with $S=\{1,2\}$).
Note that every two-element structure is obviously simple.

For uniqueness in this case, suppose that
$\mathcal{A}=\mathcal{T}[\mathcal{A}_t^\prime : t\in T]$
is another modular decomposition of $\mathcal{A}$. Let $t_1,t_2\in T$ be such that
$A_{t_1}^\prime\cap A_1\neq \emptyset$, $A_{t_2}^\prime\cap A_2\neq \emptyset$.
Suppose first that $|T|>2$, and let $t\in T\setminus\{t_1,t_2\}$.
Clearly, at least one of the sets $A_t^\prime\cap A_1$ or $A_t^\prime\cap A_2$ is non-empty;
without loss of generality we assume that
$A_t^\prime\cap A_1\neq \emptyset$. Let $U=\{ u\in T : A_u^\prime \cap A_1\neq\emptyset\}\setminus\{t_2\}$. Then it is easy to see that $U$ is an interval in $\mathcal{T}$. But $t,t_1\in U$, contradicting the simplicity of $\mathcal{T}$. Thus $T=\{t_1,t_2\}$, and it is easy to check that $\mathcal{T}\cong\mathcal{S}$ under the isomorphism $1\mapsto t_1$, $2\mapsto t_2$.

Now consider the case where no two intervals in $M$ intersect.
Clearly, now the sets in $M$ partition $A$.
For each $I\in M$ choose a representative $x_I\in I$, and define the $\mathcal{L}$-structure $\mathcal{S}$ to be the restriction
of the structure $\mathcal{A}$ to the set $S=\dom(\mathcal{S})=\{x_I: I\in M\}$.
Clearly $\mathcal{A}$ is the inflation of $\mathcal{S}$ by the corresponding intervals of $\mathcal{A}$.
The simplicity of $\mathcal{S}$ follows from the observation that if $\mathcal{S}$ contained a proper interval $K$, then $\bigcup \{I : x_I\in K\}$ would be a proper interval of $\mathcal{A}$ contradicting the definition of $M$.

To show uniqueness of $\mathcal{S}$ and all $\mathcal{A}_s$ in this case, suppose that
we have another modular decomposition
$\mathcal{A}=\mathcal{T}[\mathcal{A}_t^\prime : t\in T]$.
Since the intervals in $M$ are maximal and disjoint, each $A_t^\prime$ is contained in some member of $M$.
Suppose $I\in M$ contains several such intervals, $I=A_{t_1}^\prime\cup\ldots\cup A_{t_j}^\prime$.
It is then easy to verify that $\{ t_1,\ldots, t_j\}$ is an interval in $\mathcal{T}$,
a contradiction.
\end{proof}

The two cases in the above proof will be referred to as the \emph{degenerate} and \emph{non-degenerate} substitution decompositions. The degenerate case can give rise to non-uniqueness of this decomposition, and this may be dealt with in several ways, either by specifying the properties that one of the two substructures must possess (for example, the substructure must not itself be decomposable in the same way), or by decomposing the structure into as many pieces as possible (for example, dividing a disconnected graph into its connected components).

We use the former treatment here, but note that the latter is often better suited for certain problems, particularly in relation to the \emph{substitution decomposition tree} of a structure $\A$: each node of this tree corresponds to a substructure of $\A$ whose ground set is an interval, with the root of the tree being $\A$ and the leaves being the singleton ground sets. For a given node with corresponding non-singleton structure $\A'$, the children of $\A'$ are the substructures $\A'_s$ in the decomposition $\A'=\mathcal{S}'[\A'_s:s\in\dom(\mathcal{S})]$. The complexity of computing these trees for particular relational structures has received considerable attention in recent years, particularly in the case of graphs because of its applications in optimisation. See, for example, \cite{bergeron:computing-commo:,cournier:a-new-linear-al:,dahlhaus:efficient-and:,mcconnell:linear-time-digraphs:,mcconnell:modular-decomposition:}.

The proofs of our main theorems below all follow the same basic pattern: we take an arbitrary structure $\mathcal{A}$ of the type under consideration, identify its substitution decomposition, and then employ an inductive argument to show how adding extra points can be used to break the intervals of $\A$. The degenerate case usually needs to be considered separately, and it is precisely this case that normally gives our upper bounds.

\section{Tournaments}\label{sec-tournaments}
Recall that a {\it tournament} is an oriented complete graph.
An interval in a tournament (often called a {\it convex set} in the literature) is a set $I$ of vertices such that for every pair $x,y\in I$ and any $z\in V(T)\setminus I$, either $x\rightarrow z$ and $y\rightarrow z$ or $z\rightarrow x$ and $z\rightarrow y$.

The name ``tournament'' derives from its use to describe a competition where every pair of players $x,y$ must meet each other in a match, the outcome being either that $x$ wins, $y\rightarrow x$, or $x$ loses, $x\rightarrow y$. This viewpoint then defines an algebra with two idempotent binary operations $A_T=\langle T,\vee,\wedge\rangle$, so that if $x\rightarrow y$, then $x\vee y = y\vee x = x$ and $x\wedge y = y\wedge x = y$.
An extensive survey of tournaments from this algebraic viewpoint has been written by Crvenkovi{\'c}, Dolinka and Markovi{\'c} \cite{crvenkovic:a-survey-of-alg:}.
Of interest for us here is that a tournament is simple if and only if its corresponding abstract algebra is also simple,
(i.e. it has no proper congruences or, equivalently, homomorphic images).
Simple extensions of tournaments were studied in a string of papers in the early 1970s, and in particular it is known that
at most two additional vertices are needed:

\begin{theorem}[Erd{\H{o}}s, Fried, Hajnal and Milner~{\cite{erdos:some-remarks-on:}}]\label{tournament}
Every tournament on $n$ vertices has a simple extension with at most $2$ additional vertices.
\end{theorem}

\begin{proposition}[Erd{\H{o}}s, Hajnal and Milner~{\cite{erdos:simple-one-poin:}}]\label{tournament-one-point}
A tournament $T$ has a one-point simple extension if and only if $T$ is not finite odd chain with $5$ or more vertices.
\end{proposition}

Note that these results hold for tournaments of arbitrary cardinality, though they were originally proved for finite tournaments by Moon~\cite{moon:embedding-tourn:}. We give here another proof of the finite case using the substitution decomposition,
in order to introduce the methods we will use in the following sections, and to be able to compare those later results with the tournament case.

\begin{proof}
We proceed by induction on $|T|=n\geq 2$ by constructing two one-point extensions, $T^1$ and $T^2$, satisfying:
\begin{itemize}
\item All proper intervals of $T^1=T\cup\{x\}$ and $T^2=T\cup\{y\}$ must contain $x$ and $y$, respectively.
\item $T^1=T\cup\{x\}$, constructed as described below, is simple unless $n=3$ or $T$ is a finite odd chain.
\item $T^2=T\cup\{y\}$ is derived from $T^1$ by the rule $y\rightarrow v$ if and only if $x\leftarrow v$ for each $v\in V(T)$.
\item If $n=3$ or $T$ is a finite odd chain, the extension $T^{12}$, with edgeset $E(T^1)\cup E(T^2)\cup\{x\rightarrow y\}$, is a two-point simple extension.
\end{itemize}

The base case is $n=2$, for which we may assume $T=\{u,v\}$ with $u\rightarrow v$. We define $x$ by $v\rightarrow x\rightarrow u$, and note that $T^1$ is simple, while $T^2$ has proper intervals $\{u,y\}$ and $\{v,y\}$, both of which contain $y$.

Now suppose $n\geq 3$, and consider the substitution decomposition $T=S[A_s:s \in S]$ where the simple quotient $S$ of $T$ has $m\geq 2$ vertices. Suppose first that $m>2$, in which case the tournaments $A_s$ are defined uniquely, we construct $T^1$ as follows. First, if every $A_s$ contains a single vertex then $T=S$ is simple. Unless $|T|=3$ (a case which will be covered separately later), the addition of $x$ will preserve simplicity providing it does not have the same neighbourhoods as any existing vertex of $T$ and it does not satisfy $x\rightarrow T$ or $T\rightarrow x$.  Note that there are $2^n-n-2$ ways of choosing such an $x$.  Moreover, since $T$ was simple the addition of the vertex $y$ to form $T^2$ can introduce proper intervals only of the form $\{t,y\}$ for $t \in T$ (note that $V(T)$ cannot form an interval in $T^2$ as then it would also have formed an interval in $T^1$, a contradiction).

Thus we may assume that at least one interval $A_s$ contains more than one vertex. In this case, we add a vertex $x$ which we connect to each nonsingleton interval $A_s$ inductively, forming $A_s^1$. Now fix any nonsingleton interval $A_{s^*}$, and for each singleton interval $A_{s}$ form $A_{s}^1$ to satisfy $x\rightarrow A_{s}$ if and only if $A_{s}\rightarrow A_{s^*}$.

\begin{claim}The one-point extension $T^1$ of $T$ is simple.\end{claim}

\begin{proof}
Note first that by the inductive construction we cannot have $x\rightarrow T$ or $T\rightarrow x$ in $T^1$, otherwise the extension $A_s^1$ of the nonsingleton block $A_s$ would contradict the inductive hypothesis. Now suppose $I$ is a non-singleton interval of $T^1$, and let $u,v\in I$ be two distinct vertices. The possible positions of $u$ and $v$ in $T^1$ give rise to three cases.

\textit{Case 1:} $u\in A_s$, $v\in A_{s'}$ for some $s,s'\in S$ with $s\neq s'$. Then by the simplicity of the quotient $S$ we have $V(T)\subseteq I$, and $x$ is then included since we cannot have $x\rightarrow T$ or $T\rightarrow x$.

\textit{Case 2:} $u\in A_s$, $v=x$, for some $s\in S$. If there exists some other nonsingleton block $A_{s'}$ of $T$, then we have either $u\rightarrow A_{s'}$ or $A_{s'}\rightarrow u$, but not $v\rightarrow A_{s'}$ or $A_{s'}\rightarrow v$, reducing to Case 1. If every other $A_{s'}$ is a singleton, then $A_s$ is the only nonsingleton block so $s=s^*$, and, by the construction, $u$ and $v$ will disagree on every singleton block, again reducing to Case 1.

\textit{Case 3:} $u,v\in A_s$ for some $s\in S$. By induction, the only possible proper intervals of $A_s^1$ are those which involve $x$, reducing immediately to Case 2.
\end{proof}

Consequently, the one-point extension $T^2$, if not itself simple, contains proper intervals each of which must contain $y$, as any other interval would also be an interval in $T^1$.

\begin{figure}
\begin{center}
\begin{tabular}{ccc}
\psset{xunit=0.014in, yunit=0.014in}
\psset{linewidth=0.005in}
\begin{pspicture}(0,0)(60,50)

\Cnode*[fillstyle=solid,radius=0.05in](0,40){p1}
\Cnode*[fillstyle=solid,radius=0.05in](30,40){p2}
\Cnode*[fillstyle=solid,radius=0.05in](60,40){p3}
\Cnode[radius=0.05in](30,10){x}
\psset{arrowsize=3pt 2}
\psset{arrowinset=0.2}
\ncline{->}{p1}{p2}
\ncline{->}{p2}{p3}
\nccurve[angleA=135,angleB=45]{<-}{p3}{p1}
\ncline{<-}{p1}{x}
\ncline{->}{p2}{x}
\ncline{<-}{p3}{x}
\end{pspicture}
&\hspace{10pt}&

\psset{xunit=0.014in, yunit=0.014in}
\psset{linewidth=0.005in}
\begin{pspicture}(0,0)(60,60)

\Cnode*[fillstyle=solid,radius=0.05in](0,40){p1}
\Cnode*[fillstyle=solid,radius=0.05in](30,40){p2}
\Cnode*[fillstyle=solid,radius=0.05in](60,40){p3}
\Cnode[radius=0.05in](30,10){x}
\psset{arrowsize=3pt 2}
\psset{arrowinset=0.2}
\ncline{->}{p1}{p2}
\ncline{->}{p2}{p3}
\nccurve[angleA=135,angleB=45]{->}{p3}{p1}
\ncline{<-}{p1}{x}
\ncline{->}{p2}{x}
\ncline{<-}{p3}{x}
\end{pspicture}
\end{tabular}
\end{center}
\caption{One-vertex (non-simple) extensions of tournaments on $3$ vertices.}\label{fig-3-vertex-tournament}
\end{figure}
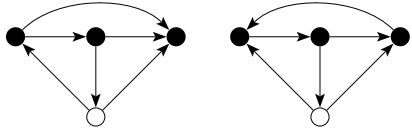

This leaves the case where $|T|=3$ or $T=S[A_1,A_2]$ with $|S|=2$, $A_1\rightarrow A_2$, and $A_1$ and $A_2$ not necessarily uniquely defined. Note that there are no simple tournaments on four vertices, and so if $|T|=3$ a simple extension will require at least two points. For the induction, however, we must create $T^1$ for which every proper interval must include the additional vertex, and this is done in Figure~\ref{fig-3-vertex-tournament}. Note then that the two-vertex extensions $T^{12}$ are simple in both cases.

Thus we may now assume that $T$ has at least four vertices, and $T$ may be written (possibly not uniquely) as $T=S[A_1,A_2]$. If $T$ is a finite odd chain, then we may pick $A_2$ to be a finite odd chain and $A_1$ a finite even chain.
Denote by $a$ the vertex of $A_1$ for which $A_1\setminus\{a\}\rightarrow a$, and by $b$ the vertex of $A_2$ for which $b\rightarrow A_2\setminus\{b\}$. We adjoin a vertex $x$ to $T$ so that $A_2\cup\{x\}$ forms whichever of $A_2^1$ or $A_2^2$ ensures $x\rightarrow b$, and $A_1\cup\{x\}$ forms the simple one-point extension $A_1^1$. The resulting extension $T^1$ is not simple, but we note that any proper interval must contain the vertex $x$, as the only other candidate we need to check is $\{a,b\}$, which, since we fixed $x\rightarrow b$, cannot be an interval unless $x\rightarrow a$, contradicting the simplicity of $A_1^1$. We observe also that $T^2$ must satisfy the same property, and consequently the two-point simple extension $T^{12}$ is simple.

Assuming herein that $T$ is not a finite odd chain, consider first the case where the decomposition $T=S[A_1,A_2]$ satisfies $|A_2|=1$.
Since $A_1$ (which contains at least three vertices) cannot be a chain, there is no $a\in A_1$ for which $A_1\setminus\{a\}\rightarrow a$. We construct $T^1=T\cup\{x\}$ so that $A_2\rightarrow x$ and so that $A_1^1=A^1\cup\{x\}$ is, by induction, a simple extension of $A_1$. It is easy to check that $T^1$ is simple, and then $T^2$ can contain only proper intervals involving $x$. A similar argument applies when $|A_1|=1$.

Finally, if neither $|A_1|=1$ nor $|A_2|=1$ then by induction we may adjoin a single vertex $x$ for which $A_1\cup\{x\}$ corresponds to the simple extension $A_1^1$ and $A_2\cup\{x\}$ to $A_2^1$. We claim that the tournament $T\cup\{x\}$ gives us $T^1$. Since $A_1^1$ and $A_2^1$ are simple one-point extensions, we need only consider intervals of the form $\{a,b\}$ with $a\in A_1$ and $b\in A_2$. Since $a\rightarrow A_2$, this would imply $b\rightarrow A_2\setminus\{b\}$ from which we can conclude that $x\rightarrow b$, as $A_2^1$ is simple. Similarly, $A_1\rightarrow b$ implies $A_1\setminus\{a\}\rightarrow a$ and hence (by the simplicity of $A_1^1$) $a\rightarrow x$, so $a$ and $b$ disagree on $x$, completing the proof.
\end{proof}

It is initially surprising that the case where $T$ is a finite odd chain with $5$ or more vertices requires two additional vertices rather than one.  Figure~\ref{fig-transitive-tournament} shows a two-point simple extension of a chain on $7$ vertices.

\begin{figure}
\begin{center}
\psset{xunit=0.014in, yunit=0.014in}
\psset{linewidth=0.005in}
\begin{pspicture}(0,0)(120,60)

\Cnode*[fillstyle=solid,radius=0.05in](0,50){p1}
\Cnode*[fillstyle=solid,radius=0.05in](20,50){p2}
\Cnode*[fillstyle=solid,radius=0.05in](40,50){p3}
\Cnode*[fillstyle=solid,radius=0.05in](60,50){p4}
\Cnode*[fillstyle=solid,radius=0.05in](80,50){p5}
\Cnode*[fillstyle=solid,radius=0.05in](100,50){p6}
\Cnode*[fillstyle=solid,radius=0.05in](120,50){p7}
\Cnode[radius=0.05in](40,10){x}
\Cnode[radius=0.05in](80,10){y}
\psset{arrowsize=3pt 2}
\psset{arrowinset=0.2}
\ncline{->}{p1}{p2}
\ncline{->}{p2}{p3}
\ncline{->}{p3}{p4}
\ncline{->}{p4}{p5}
\ncline{->}{p5}{p6}
\ncline{->}{p6}{p7}
\ncline{<-}{p1}{x}
\ncline{->}{p2}{x}
\ncline{<-}{p3}{x}
\ncline{->}{p4}{x}
\ncline{<-}{p5}{x}
\ncline{->}{p6}{x}
\ncline{<-}{p7}{x}
\ncline{->}{p1}{y}
\ncline{<-}{p2}{y}
\ncline{->}{p3}{y}
\ncline{<-}{p4}{y}
\ncline{->}{p5}{y}
\ncline{<-}{p6}{y}
\ncline{->}{p7}{y}
\ncline{->}{x}{y}
\end{pspicture}
\end{center}
\caption{A $2$-vertex simple extension of a transitive tournament on $7$ vertices.}\label{fig-transitive-tournament}
\end{figure}
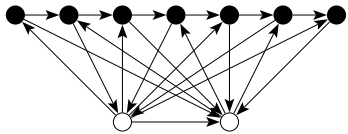


\section{Graphs}\label{sec-graphs}

In a graph $G$, an interval $I$ is a set of vertices for which
$N(u)\setminus I = N(v)\setminus I$ for every $u,v\in I$. A graph is
simple if it contains no proper intervals, but it is worth noting
that the word ``simple'' is commonly used to describe graphs with no multiple edges or loops, and our
notion is more often referred to as \emph{prime} or
\emph{indecomposable}. These graphs have been the subject of considerable study, see for example Ehrenfeucht, Harju, and Rozenberg~\cite{ehrenfeucht:the-theory-of-2:}, Ille~\cite{ille:indecomposable-:}, and Sabidussi~\cite{sabidussi:graph-derivativ:}. A survey of indecomposability and the substitution decomposition in graphs can be found in Brandst{\"a}dt, Le, and Spinrad's text~\cite{brandstadt:graph-classes:-:}. Simple extensions of graphs have received some attention in the past. In particular, we have:

\begin{lemma}[Sumner~{\cite[Theorem 2.45]{sumner:indecomposable-:}}]\label{completegraph}
The complete graph $K_n$ has a simple extension with $\lceil\log_2(n+1)\rceil$
additional vertices.
\end{lemma}

The bound $m=\lceil\log_2(n+1)\rceil$ is also the smallest
possible, for were we to add a set $B$ of $m$ vertices with $n>2^m-1$, then either two vertices in $G$ have the same neighbourhood in $G\cup B$, or one vertex of $G$ is connected to
every other vertex in $G\cup B$, both of which give an interval.

\begin{theorem}\label{thm-graphs}
Every graph $G$ has a simple extension with at most
$m=\lceil\log_2(|V(G)|+1)\rceil$ additional vertices.
\end{theorem}

\begin{proof} First note that when $|G|=1$, we may trivially extend by 1 vertex to form $K_2$. We will prove, by induction on $|G|=n\geq 2$ the following statement: $G$ has a simple extension with a set $B$ of $m$ independent vertices\footnote{The assumption that $B$ is independent may be replaced by the requirement that $B$ be any graph specified in advance, but this adds considerably to the complexity and length of the proof.}, where $2\leq m\leq \lceil\log_2(n+1)\rceil$.

The base case $n=2$ may be seen in Figure~\ref{fig-two-point-graph}, so now suppose $n\geq 3$. Write $G = H[J_v:v\in V(H)]$ where $H$ is the simple quotient of $G$, and suppose first that $|H|>2$. If $H=G$ then the graph is already simple, but for the induction we need to add two vertices. As in the tournament case, we may preserve simplicity by adding vertices whose neighbourhoods in $G$ are distinct and neither the same as any vertex in $G$, empty, nor all of $V(G)$. As in the case of tournaments, are $2^n-n-2$ possible neighbourhoods, and for $n\geq 4$ we may pick any two of these for the new vertices.

Thus we may now suppose that $G\neq H$. By induction, each $J_v$ of size $|J_v|\geq 2$ can be extended by some independent set $B_v$ of size at most $\lceil\log_2(|V(J_v)|+1)\rceil$ so that $J_v\cup B_v$ is simple. Picking $x\in V(H)$ for which $B_x=B$ is largest, for every $v\neq x$ in $H$ with $|J_v|\geq 2$ we identify $B_v$ with any subset of $B$ of the correct size. For any $v\in H$ for which $|J_v|=1$, we connect $J_v$ to any vertex in $B$ that is not adjacent to all of $J_x$, and label this singleton set $B_v$.

\begin{figure}
\begin{center}
\begin{tabular}{ccc}
\psset{xunit=0.014in, yunit=0.014in}
\psset{linewidth=0.005in}
\begin{pspicture}(0,0)(40,40)
\Cnode*[fillstyle=solid,radius=0.04in](10,30){a}
\Cnode*[fillstyle=solid,radius=0.04in](30,30){b}
\Cnode[radius=0.04in](10,10){x}
\Cnode[radius=0.04in](30,10){y}
\ncline{-}{x}{a}
\ncline{-}{a}{y}
\ncline{-}{y}{b}
\end{pspicture}
&\hspace{10pt}&
\psset{xunit=0.014in, yunit=0.014in}
\psset{linewidth=0.005in}
\begin{pspicture}(0,0)(40,40)
\Cnode*[fillstyle=solid,radius=0.04in](10,30){a}
\Cnode*[fillstyle=solid,radius=0.04in](30,30){b}
\Cnode[radius=0.04in](10,10){x}
\Cnode[radius=0.04in](30,10){y}
\ncline{-}{a}{b}
\ncline{-}{x}{a}
\ncline{-}{b}{y}
\end{pspicture}
\end{tabular}
\end{center}
\caption{Two point simple extensions in the cases when $|G|=2$. In each case the filled nodes correspond to $G$, and the empty nodes to $B$.}\label{fig-two-point-graph}
\end{figure}
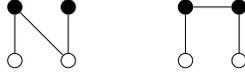

\begin{claim}
The extension $G\cup B$ is simple.
\end{claim}

\begin{proof}
Consider an interval $I$ with at least two vertices $a$ and $b$. There are four cases:

\textit{Case 1:} $a\in J_u$, $b\in J_v$ for $u\neq v$ in $H$. The simplicity of $H$ implies that $G\subseteq I$. In particular this gives $J_x\subseteq I$, which implies $B\subseteq I$ since $J_x\cup B$ is simple.

\textit{Case 2:} $a,b \in J_v\cup B_v$, $v\neq x$. By the simplicity of $J_v\cup B_v$ we have $J_v \cup B_v\subseteq I$. If $|J_v|\geq 2$ then by induction $|B_v|\geq 2$ implying that $J_x\subseteq I$, reducing to Case 1. If, say, $J_v=\{a\}$ then either $a$ is adjacent to everything in $J_x$ or nothing in $J_x$, while by the construction $b\in B_v$ is adjacent to some but not all of $J_x$, reducing to Case 1.

\textit{Case 3:} $a,b\in J_x\cup B$. We immediately have $J_x\cup B \subseteq I$, and so if there is some $v\neq x$ for which $|J_v|\geq 2$ then $J_v\subseteq I$ gives Case 1, so we now assume $|J_v|=1$ for all $v\neq x$ in $H$. Pick any such $J_v=\{c\}$ we know there exist $b_1,b_2\in B$ with $b_1\sim c$ but $b_2\not\sim c$, implying $c\in I$, which again reduces to Case 1.

\textit{Case 4:} $a\in J_v$, $b\in B\setminus B_v$. Since $J_v\cup B_v$ is simple, we know that $a$ is adjacent to some $c\in J_v\cup B_v$, while $b$ is not, reducing to Case 2.
\end{proof}

For the degenerate cases, let us first assume that $H=\overline{K_2}$, i.e. that
$G$ is disconnected, so $G=H[J_1,J_2]$ where $J_1$ and $J_2$ may be picked possibly in a number of ways. We arrange it so that $J_1$ is a largest connected component, and suppose first that $|J_1|\geq 2$. By induction, each of $J_1$ and $J_2$ may be extended by adding a set of vertices $B_1$ and $B_2$, respectively, unless $|J_2|=1$, where we pick any vertex of $B$ to act as the singleton set $B_2$. Fix $x\in \{1,2\}$ so that $B=B_x$ is the larger of $B_1$ and $B_2$, and associate the other set $B_{x'}$, $x'\in\{1,2\}$, $x'\neq x$, with any subset of $B$. If $|J_1|=1$ then $G$ is in fact $\overline{K_n}$, and but for the induction we could appeal to Lemma~\ref{completegraph}. By induction, there exists a set $B_2$ of $\lceil\log_2 n\rceil$ vertices for which $J_2\cup B_2$ is a simple extension of $J_2$. Unless $\lceil\log_2 (n+1)\rceil = \lceil\log_2 n\rceil +1$, we may assign a nonempty neighbourhood of $B_2$ to the single vertex of $J_1$ that is different to the neighbourhood of every vertex in $J_2$. In the exceptional case where $\lceil\log_2 (n+1)\rceil = \lceil\log_2 n\rceil +1$, we assign any nonempty neighbourhood of $B_2$ to $J_1$, observing that there exists a vertex $v\in J_2$ with the same neighbourhood. However, we are now permitted to add another new vertex $b$, attached to $J_1$ and nothing else, so that $v$ no longer has the same neighbourhood as the single vertex in $J_1$.

In the case that $H=K_2$, only a few modifications need to be made to the above. We first pick $J_1$ so that $\overline{J_1}$ is the largest possible connected component. When $J_1$ is not a singleton but $J_2$ is, we select $B_2=\{b\}$ by choosing any $b\in B$ for which $b$ is not connected to all of $J_1$ (as in the non-degenerate case). When $J_1$ is a singleton (so $G$ is a complete graph), it can have any non-complete neighbourhood of $B=B_2$ (including the empty neighbourhood), providing no vertex of $J_2$ already has that neighbourhood. If no such neighbourhood exists, then we may add a new vertex to $B$ and connect it to $J_1$. The analysis of these degenerate cases is similar to and easier than the analysis in the nondegenerate case above.
\end{proof}

Although this bound is tight for complete and independent graphs, the inductive construction used in the proof actually does somewhat better in many cases. Unless the graph is complete or empty, at each stage of the induction we add only as many vertices as are required by any single block $J_v$ of the substitution decomposition. This may be thought of iteratively in terms of the substitution decomposition tree: recalling that the degenerate cases are handled in this tree by breaking up the graph into as many parts as possible, the largest such degenerate node all of whose children are leaves will dictate how many additional vertices are needed. If no such degenerate node exists, then the inductive construction above requires only two additional vertices.


\section{Permutations}\label{sec-permutations}

In a permutation $\pi$ of $[n]$ an interval is a collection of entries that are
contiguous both by position and value. If a permutation is viewed as a set with two linear orders $<$ and $\prec$, an interval is a set of points that is contiguous under both
$<$ and $\prec$. Intervals are  easily identified in the plot of a permutation as sets of points enclosed in an axis-parallel rectangle, with no points lying in the regions above, below, to the left or to the right (see Figure~\ref{fig-permutation-interval}). To embed a given $\pi$ in a simple permutation, therefore, we must ensure that every axis-parallel rectangle containing at least two points of $\pi$
is ``broken'' by at least one of the extension points.

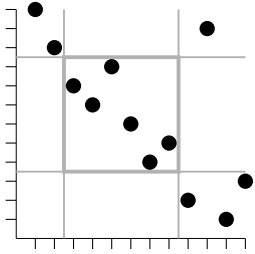
\begin{figure}
\begin{center}
\psset{xunit=0.01in, yunit=0.01in} \psset{linewidth=0.005in}
\begin{pspicture}(0,0)(120,120)
\psaxes[dy=10,dx=10,tickstyle=bottom,showorigin=false,labels=none](0,0)(120,120)
\psframe[linecolor=darkgray,linewidth=0.02in](24,34)(86,96)
\psline[linecolor=darkgray,linewidth=0.01in](25,0)(25,120)
\psline[linecolor=darkgray,linewidth=0.01in](85,0)(85,120)
\psline[linecolor=darkgray,linewidth=0.01in](0,35)(120,35)
\psline[linecolor=darkgray,linewidth=0.01in](0,95)(120,95)

\pscircle*(10,120){0.04in}
\pscircle*(20,100){0.04in}
\pscircle*(30,80){0.04in}
\pscircle*(40,70){0.04in}
\pscircle*(50,90){0.04in}
\pscircle*(60,60){0.04in}
\pscircle*(70,40){0.04in}
\pscircle*(80,50){0.04in}
\pscircle*(90,20){0.04in}
\pscircle*(100,110){0.04in}
\pscircle*(110,10){0.04in}
\pscircle*(120,30){0.04in}
\end{pspicture}
\end{center}
\caption{An interval in the plot of a permutation.}
\label{fig-permutation-interval}
\end{figure}

\begin{theorem}\label{thm-permutations}
Every permutation $\pi$ on $n$ symbols has a simple extension with at most $\lceil(n+1)/2\rceil$ additional points.
\end{theorem}

\begin{proof}
We proceed by induction on $n\geq 2$, claiming that for each permutation $\pi$ of length $n$ we may construct two extensions $\pi^{\uparrow}$ and $\pi^{\downarrow}$ both containing at most
$\lceil \frac{n+1}{2}\rceil$ new points, satisfying:
\begin{itemize}
\item Viewed as extensions, each $\pi^{a}$ ($a\in\{\uparrow,\downarrow\}$) has a new leftmost point
 which is neither a new maximum nor a new minimum, called the \emph{entry point} and denoted $\In(\pi^a)$.
\item Each $\pi^a$ ($a\in\{\uparrow,\downarrow\}$) has a new \emph{exit point} $\Out(\pi^a)$;
for $\pi^{\uparrow}$ this is a new maximum while for $\pi^{\downarrow}$ this is a new minimum, and in both cases it is neither a leftmost nor a rightmost point.
\item Any proper interval of $\pi^{a}$ contains the exit point $\Out(\pi^a)$, but does not contain the rightmost point of $\pi$.
\item At least one of $\pi^{\uparrow}$ and $\pi^{\downarrow}$ is simple.
\end{itemize}
Let us call the permutation obtained from $\pi^a$ by removing the entry and exit points the \emph{core} of $\pi^a$ and denote it by $\Core(\pi^a)$. Note that $\Core(\pi^a)$ still contains the original copy of $\pi$.

In the base case $n=2$, either $\pi=12$ or $\pi=21$. When $\pi=12$, $\pi^{\uparrow}=2413$ is simple, and the only non-singleton intervals of $\pi^{\downarrow}=3124$ are $12$ and $312$, both of which contain the exit point. The case $\pi=21$ is dealt with by symmetry.

Now suppose $n\geq 3$. Write $\pi = \sigma [\pi_1,\pi_2,\ldots,\pi_m]$ where the simple quotient $\sigma$ of $\pi$ is of length $m\geq 2$, and $\pi_1,\pi_2,\ldots, \pi_m$ are permutations. First suppose $m>2$ so that the substitution decomposition is unique. If $\pi_i = 1$ for all $i$, then $\pi = \sigma$ is already simple. We construct $\pi^{\uparrow}$ and $\pi^{\downarrow}$ by adding precisely two points. The first is a new leftmost point, which may be positioned vertically anywhere except as a new maximum, minimum, or adjacent to the leftmost entry of $\pi$. The new maximum or minimum is inserted similarly, and it is easily checked that the resulting permutation is simple.

Now suppose that at least one $\pi_i$ contains at least two points.
The permutations $\pi^\uparrow$ and $\pi^\downarrow$ will be constructed by extending
each non-trivial block $\pi_i$ of $\pi$ to one of $\pi_i^\uparrow$ or $\pi_i^\downarrow$ and then `repositioning' the entry and exit points so that they `link' successive blocks.
The choice between $\pi_i^\uparrow$ and $\pi_i^\downarrow$ is made so as to enable this linking.
More formally, to define $\pi^a$ ($a\in\{\uparrow,\downarrow\}$), first write $\sigma=s_1\ldots s_m$,
and suppose that the non-trivial blocks of $\pi$ are, from left to right, $\pi_{i_1},\ldots,\pi_{i_t}$.
Next, for each $i=1,\ldots,m$ define
\[
\tau_i=\left\{ \begin{array}{ll}
\Core(\pi_i^\uparrow) & \mbox{if } i=i_r,\ 1\leq r\leq t-1, \mbox{ and } s_{i_r}<s_{i_{r+1}}\\
\Core(\pi_i^\downarrow) & \mbox{if } i=i_r,\ 1\leq r\leq t-1, \mbox{ and } s_{i_r}>s_{i_{r+1}}\\
\Core(\pi_i^a) & \mbox{if } i=i_t\\
1 & \mbox{otherwise,}
\end{array} \right.
\]
and form the permutation $\tau = \sigma[\tau_1,\ldots,\tau_m]$.
For each $r=1,\ldots ,t-1$ insert a linking point $\ell_r$, the position of which is uniquely determined by the requirement that $\ell_r$ is the exit point for $\tau_{i_r}$ and the entry point for $\tau_{i_{r+1}}$.
Finally, complete the construction of $\pi^a$ by defining $\ell_0=\In(\pi^a)$ to be the leftmost point, which at the same time serves as the entry point to $\tau_{i_1}$, while $\ell_t=\Out(\pi^a)$ is a new minimum or a new maximum (depending on $a$) which is at the same time the exit point for $\tau_{i_t}$.
This construction is illustrated in Figure~\ref{fig-permutation-extension}.

\begin{figure}
\begin{center}
\begin{tabular}{ccc}
\psset{xunit=0.013in, yunit=0.013in} \psset{linewidth=0.005in}
\begin{pspicture}(-10,-10)(160,160)
\psaxes[dy=10,dx=10,tickstyle=bottom,showorigin=false,labels=none](-10,-10)(150,150)
{\psset{fillstyle=none,linecolor=black,cornersize=absolute,linearc=3pt}
  \psframe[linewidth=0.005in,linestyle=solid](6,26)(34,54)
  \psframe[linewidth=0.005in,linestyle=solid](46,66)(84,104)
  \psframe[linewidth=0.005in,linestyle=solid](86,6)(104,24)
  \psframe[linewidth=0.005in,linestyle=solid](116,106)(144,134)}
\pscircle*(40,140){0.03in}
\pscircle*(110,60){0.03in}
{\small
\rput[c](20,40){$\pi_1$}
\rput[c](65,85){$\pi_3$}
\rput[c](95,15){$\pi_4$}
\rput[c](130,120){$\pi_6$}}
\end{pspicture}
&\rule{10pt}{0pt}&
\psset{xunit=0.013in, yunit=0.013in}
\psset{linewidth=0.005in}
\begin{pspicture}(-30,-10)(150,150)
\psaxes[dy=10,dx=10,tickstyle=bottom,showorigin=false,labels=none](-10,-10)(150,150)
\psline(-5,45)(25,45)(25,95)(65,95)(65,15)(95,15)(95,115)(125,115)(125,145)
{\psset{fillstyle=solid,fillcolor=white}
  \psframe[linewidth=0.005in,linestyle=solid](3,23)(37,57)
  \psframe[linewidth=0.005in,linestyle=solid](43,63)(87,107)
  \psframe[linewidth=0.005in,linestyle=solid](83,3)(107,27)
  \psframe[linewidth=0.005in,linestyle=solid](113,103)(147,137)
\psset{linecolor=gray,cornersize=absolute,linearc=3pt}
  \psframe[linewidth=0.005in,linestyle=solid](6,26)(34,54)
  \psframe[linewidth=0.005in,linestyle=solid](46,66)(84,104)
  \psframe[linewidth=0.005in,linestyle=solid](86,6)(104,24)
  \psframe[linewidth=0.005in,linestyle=solid](116,106)(144,134)}
{\psset{fillstyle=solid,fillcolor=white}
\pscircle(-5,45){0.03in}
\pscircle(25,95){0.03in}
\pscircle(65,15){0.03in}
\pscircle(95,115){0.03in}
\pscircle(125,145){0.03in}}
\pscircle*(40,140){0.03in}
\pscircle*(110,60){0.03in}
{\small
\rput[c](20,40){$\tau_1^{\uparrow}$}
\rput[c](65,85){$\tau_3^{\downarrow}$}
\rput[c](95,15){$\tau_4^{\uparrow}$}
\rput[c](130,120){$\tau_6^{\uparrow}$}
\rput[r](-12,45){$\In(\pi^\uparrow)$}
\rput[r](22,102){$\ell_1$}
\rput[r](62,11){$\ell_2$}
\rput[r](92,122){$\ell_3$}
\rput[c](125,155){$\Out(\pi^\uparrow)$}
}
\end{pspicture}
\end{tabular}
\end{center}
\caption{On the left, the substitution decomposition of $\pi=264135[\pi_1,1,\pi_3,\pi_4,1,\pi_6]$. On the right, the extension $\pi^\uparrow$ formed from the blocks $\tau_i^a$ connected by linking points $\ell_1,\ell_2$ and $\ell_3$.}
\label{fig-permutation-extension}
\end{figure}
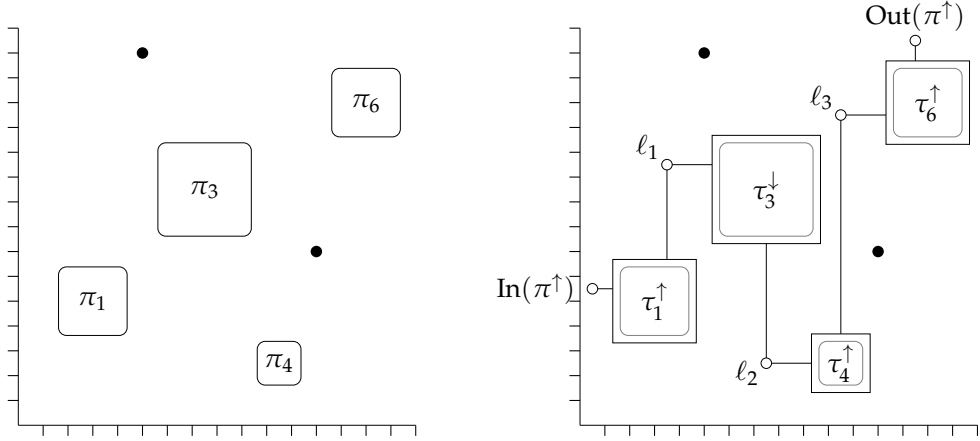

\begin{claim}
If $\pi^\uparrow$ is not simple then the rightmost non-singleton block $\pi_{i_t}$ is the maximal block by value,
and every non-trivial proper interval of $\pi^\uparrow$ contains $\Out(\pi^\uparrow)$.
\end{claim}

\begin{proof}
Let $I$ be a non-singleton interval of $\pi^\uparrow$, and let $u,v\in I$ be two distinct points.
We consider several cases, depending on the position of $u$ and $v$ in $\pi^\uparrow$.

\textit{Case 1:} $u\in \tau_j$, $v\in \tau_k$, $j\neq k$.
The simplicity of $\sigma$ implies that $I$ contains all $\tau_i$ ($1\leq i\leq m$).
Each of the remaining points $\ell_0,\ell_1,\ldots,\ell_t$ overlaps at least one
$\tau_i$, and so these points must belong to $I$ too. We conclude that $I=\pi^\uparrow$.
\vspc

\textit{Case 2:} $u,v\in\tau_{i_r}$ for some $r\in\{1,\ldots,t-1\}$.
The set $I\cap(\tau_{i_r}\cup\ell_{r-1}\cup\ell_r)$ is an interval of the permutation $\pi_{i_r}^b = \tau_{i_r}\cup\ell_{r-1}\cup\ell_r$ ($b\in\{\uparrow,\downarrow\}$), and so by induction it contains the exit point $\ell_r$.
But $\ell_r$ overlaps horizontally with $\tau_{i_{r+1}}$, and so $I$ must contain a point from $\tau_{i_{r+1}}$.
Thus we have reduced this case to Case 1.
\vspc

\textit{Case 3:}  $u,v\in\tau_{i_t}$. As in Case 2 we see that $I$ must contain the corresponding exit point,
which this time is $\Out(\pi^\uparrow)$. Now, if $\pi_{i_t}$ is the maximal block by value there is nothing further to prove, while if $\pi_{i_t}$ is not maximal we see that $I$ must contain a point from some other block, and reduce again to Case 1.
\vspc

\textit{Case 4:}
$u\in \{\ell_0,\ell_1,\ldots,\ell_{t}\}$, $v\in \tau_j$.
From the construction it follows that $u$ and $v$ are separated either horizontally or vertically by at least one other point belonging to some $\tau_k$. Thus this case reduces to one of the first three cases.
\vspc

\textit{Case 5:}
$u,v \in \{\ell_0,\ell_1,\ldots,\ell_{t}\}$.
Again, we can see that $I$ must contain at least one element from some $\tau_j$, reducing to Case 4.
\end{proof}

Returning to the proof of the theorem, we note that the analogous assertion holds for $\pi^\downarrow$:
If $\pi^\downarrow$ is not simple then the rightmost non-singleton block $\pi_{i_t}$ is the minimal block by value,
 and every non-trivial proper interval of $\pi^\downarrow$ contains $\Out(\pi^\downarrow)$. Since $\sigma$ is simple and of length at least $4$, neither the maximal nor the minimal blocks by value can be the rightmost block of $\pi$, and so the rightmost point of $\pi$ cannot appear in any proper interval of $\pi^a$. Moreover, as $\pi_{i_t}$ cannot be at the same time maximal and minimal, we conclude that at least one of $\pi^\uparrow$
or $\pi^\downarrow$ is simple.

To complete the proof in the non-degenerate case, let us estimate the number of points we have added to $\pi$.
To each non-trivial block $\pi_{i_r}$ ($1\leq r\leq t$) we have added at most $\lceil (|\pi_{i_r}|+1)/2\rceil$ points.
Moreover, we merged $t-1$ pairs of exit/entry points. Thus the total number of additional points does not exceed
\[
\sum_{r=1}^t \left\lceil \frac{|\pi_{i_r}|+1}{2}\right\rceil - (t-1) \leq \sum_{r=1}^t \frac{|\pi_{i_r}|+2}{2} - (t-1)=
1+\sum_{r=1}^t \frac{|\pi_{i_r}|}{2}\leq 1 + \frac{n}{2}.
\]
Since the number of points is a whole number, we conclude that it does not exceed \mbox{$\lceil (n+1)/2\rceil$}, as required.

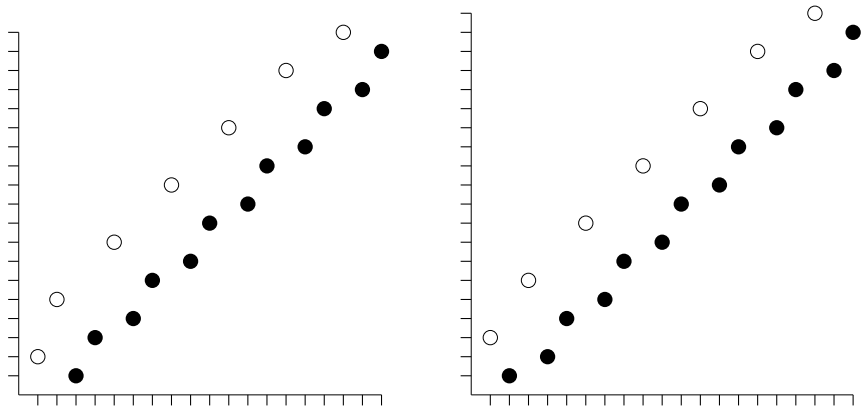
\begin{figure}
\begin{center}
\begin{tabular}{ccc}
\psset{xunit=0.01in, yunit=0.01in}
\psset{linewidth=0.005in}
\begin{pspicture}(0,0)(190,190)
\psaxes[dy=10,Dy=1,dx=10,Dx=1,tickstyle=bottom,showorigin=false,labels=none](0,0)(190,190)
\pscircle(10,20){0.04in}
\pscircle(20,50){0.04in}
\pscircle*(30,10){0.04in}
\pscircle*(40,30){0.04in}
\pscircle(50,80){0.04in}
\pscircle*(60,40){0.04in}
\pscircle*(70,60){0.04in}
\pscircle(80,110){0.04in}
\pscircle*(90,70){0.04in}
\pscircle*(100,90){0.04in}
\pscircle(110,140){0.04in}
\pscircle*(120,100){0.04in}
\pscircle*(130,120){0.04in}
\pscircle(140,170){0.04in}
\pscircle*(150,130){0.04in}
\pscircle*(160,150){0.04in}
\pscircle(170,190){0.04in}
\pscircle*(180,160){0.04in}
\pscircle*(190,180){0.04in}
\end{pspicture}
&\rule{10pt}{0pt}&
\psset{xunit=0.01in, yunit=0.01in}
\psset{linewidth=0.005in}
\begin{pspicture}(0,0)(200,200)
\psaxes[dy=10, Dy=1, dx=10, Dx=1, tickstyle=bottom, showorigin=false, labels=none](0,0)(200,200)
\pscircle(10,30){0.04in}
\pscircle*(20,10){0.04in}
\pscircle(30,60){0.04in}
\pscircle*(40,20){0.04in}
\pscircle*(50,40){0.04in}
\pscircle(60,90){0.04in}
\pscircle*(70,50){0.04in}
\pscircle*(80,70){0.04in}
\pscircle(90,120){0.04in}
\pscircle*(100,80){0.04in}
\pscircle*(110,100){0.04in}
\pscircle(120,150){0.04in}
\pscircle*(130,110){0.04in}
\pscircle*(140,130){0.04in}
\pscircle(150,180){0.04in}
\pscircle*(160,140){0.04in}
\pscircle*(170,160){0.04in}
\pscircle(180,200){0.04in}
\pscircle*(190,170){0.04in}
\pscircle*(200,190){0.04in}

\end{pspicture}
\end{tabular}
\end{center}
\caption{Simple extensions of $12\cdots n$, for $n=12$ (left) and $n=13$ (right).}
\label{fig-increasing-simples}
\end{figure}

In the case where $m=2$ we may assume without loss that $\sigma=12$, and  write $\pi=12[\pi_1,\pi_2]$,
where $\pi_1$ and $\pi_2$ may be chosen possibly in a number of different ways. Fix $\pi_1$ so that if possible it does not end with its largest element. If this is possible, and if $\pi_2$ is not a singleton, we may construct $\pi^a$ ($a\in\{\uparrow,\downarrow\}$) by extending $\pi_1$ to $\pi_1^{\uparrow}$ and $\pi_2$ to $\pi_2^a$, and identifying the exit point of the first to the entry point of the second. If $\pi_2$ is a singleton, then the exit point of $\pi_1^{\uparrow}$ may be placed above $\pi_2$ to form $\pi^\uparrow$ without producing a proper interval between the top of $\pi_1$ and $\pi_2$: the intersection of any such interval with $\pi_1^\uparrow$ would have to form a proper interval in $\pi_1^\uparrow$ containing the rightmost point of $\pi_1$, which is impossible by the inductive hypothesis. We form $\pi^{\downarrow}$ simply by taking $12[\pi_1^{\downarrow},\pi_2]$, with the exit point of $\pi_1$ acting as the exit point of $\pi$. Note that $\pi^\downarrow$ is not simple since $\pi_1^\downarrow$ is an interval, but it contains $\Out(\pi^\downarrow)$ and not the single point of $\pi_2$ and so satisfies the inductive hypothesis.

If it is not possible to pick $\pi_1$ so that it does not end with its largest element, then we first attempt to pick $\pi_2$ so that it does not start with its smallest element. If this is possible, and if $\pi_1$ is not a singleton, an analogous argument to that given above applies. On the other hand, if $\pi_1$ is a singleton, place all additional points of $\pi_2^a$ that lie to the left of every point in $\pi_2$ to the left of the singleton $\pi_1$. (This happens, for example, if $\pi=132$, and we form $\pi^\uparrow=35142$ and $\pi^\downarrow=41253$.) If it is not possible to pick $\pi_2$ so that it does not start with its smallest element, then $\pi=12\cdots n$. If $n>3$, we set $\pi_1=12\cdots n-2$ so that $\pi_2$ forms the pattern $12$, whence there is a two-point extension $\pi_2^a$, $a \in\{\uparrow,\downarrow\}$, in which both the entry and exit points occur to the left of $\Core(\pi_2)$. This permits us to connect the exit point of the extension $\pi_1^\uparrow$ to the entry point of $\pi_2^a$, and the exit point of $\pi_2^a$ both acts as the exit point of $\pi^a$ and ensures that there are no intervals between $\pi_1$ and $\pi_2$. Finally, when $\pi=123$, we set $\pi^\uparrow=31524$ and $\pi^\downarrow=42135$. In each case, checking that the so constructed extensions $\pi^\uparrow$ and $\pi^\downarrow$ satisfy the required conditions is similar to (and easier than) the non-degenerate case.
\end{proof}

See Figure~\ref{fig-increasing-simples} for examples of extensions of monotone permutations. Note that in this case the bound is tight: every adjacent pair $i,i+1$ must be ``separated'' either horizontally or vertically by one of the additional points, and the points $\pi(1)=1$ and $\pi(n)=n$ of $\pi$ must not lie in the ``corners'' of the simple extension --- a total of $n+1$ gaps to be filled, and each additional point can be used to fill at most two of these gaps (one horizontally, one vertically).


\section{Posets}\label{sec-posets}

Although posets are very naturally described as relational structures, we have delayed considering their simple extensions until now,
as with posets we encounter a ``mix'' of the results related to the degenerate decompositions corresponding to permutations and graphs.
For permutations recall that the degenerate cases correspond to the increasing and decreasing permutations, which (viewing them as relational structures) occur when the two linear orders are equal or are the opposites of each other.
For graphs, the non-uniqueness comes in the form of complete and independent graphs. Posets can be decomposed non-uniquely either through
total comparability (linear orders) or total incomparability (antichains), and the simple extension in each case is significantly different.

We begin with a result that holds for all posets; we will discuss special cases where fewer additional elements are required afterwards. Our approach takes much the same form as the permutation case, inductively identifying entry and exit points from the simple extensions of the intervals in the substitution decomposition, and this gives rise to the same bound. Because of this similarity, we will omit some of the case by case details in the proof.

\begin{theorem}A poset $(P,<)$ on $n$ elements has a simple extension with at most \mbox{$\lceil(n+1)/2\rceil$} additional elements.
\end{theorem}

\begin{proof}
Extending a poset of size 1 is trivial, so we proceed by induction on $n$ assuming $n\geq 2$. Our claim is that we may form four extensions
$P^{\uparrow}$, $P^{\downarrow}$, $P^{\updown}$ and $P^{\downup}$ of a poset $(P,<)$, satisfying:

\begin{itemize}
\item
Each extension $P^{a}$ ($a\in\{ \uparrow,\downarrow,\updown,\downup\}$) has at most $\lceil (n+1)/2\rceil$ new elements.
\item
Among these new points lie two \emph{distinguished elements}, $\Ext_1(P^a)$ and $\Ext_2(P^a)$. For $P^{\downarrow}$ these are both new minima, for $P^{\uparrow}$ new maxima, for $P^{\updown}$, the point $\Ext_1(P^\updown)$ is a maximum and $\Ext_2(P^\downup)$ is a minimum, while for $P^{\downup}$, the point $\Ext_1(P^\downup)$ is a minimum and $\Ext_2(P^\downup)$ is a maximum. Note that $P^\updown$ and $P^\downup$ need not be different extensions --- it is only the labelling of the distinguished elements that matters. The remaining points
of $P^a$ will be called the \emph{core}, denoted $\Core(P^a)$.
\item Every proper interval of $P^{a}$ ($a\in\{\uparrow,\downarrow,\updown,\downup\}$) contains $\Ext_2(P^a)$. Additionally, in $P^\uparrow$ and $P^\updown$ there is a minimal element that is not contained in any proper interval, and in $P^\downarrow$ and $P^\downup$ there is a maximal element not contained in any proper interval.
\item At least one $P^{a}$ is simple.
\end{itemize}

The base case is $n=2$, when the poset is either linear or an antichain. The extensions $P^{a}$ may be seen in Figure~\ref{fig-two-point-posets}, and it is straightforward to check that these satisfy the inductive hypothesis.  Now suppose that $n> 2$ and decompose the poset as $P=S[A_s:s\in S]$. Assuming first that $|S|>2$, we proceed in essentially the same way as in the permutation case.

\begin{figure}
\begin{center}
\begin{tabular}{ccccccc}
\psset{xunit=0.014in, yunit=0.014in}
\psset{linewidth=0.005in}
\begin{pspicture}(0,0)(40,60)
\Cnode*[fillstyle=solid,radius=0.04in](10,20){a}
\Cnode*[fillstyle=solid,radius=0.04in](10,40){b}
\Cnode[radius=0.04in](30,40){x}
\Cnode[radius=0.04in](30,60){y}
\ncline{-}{a}{b}
\ncline{-}{x}{a}
\ncline{-}{y}{b}
{\small
\rput[l](34,36){$1$}
\rput[l](34,64){$2$}}
\end{pspicture}
&\hspace{5pt}&
\psset{xunit=0.014in, yunit=0.014in}
\psset{linewidth=0.005in}
\begin{pspicture}(0,0)(40,60)
\Cnode*[fillstyle=solid,radius=0.04in](10,20){a}
\Cnode*[fillstyle=solid,radius=0.04in](10,40){b}
\Cnode[radius=0.04in](30,0){x}
\Cnode[radius=0.04in](30,20){y}
\ncline{-}{a}{b}
\ncline{-}{x}{a}
\ncline{-}{y}{b}
{\small
\rput[l](34,-4){$2$}
\rput[l](34,24){$1$}}
\end{pspicture}
&\hspace{5pt}&
\psset{xunit=0.014in, yunit=0.014in}
\psset{linewidth=0.005in}
\begin{pspicture}(0,0)(40,60)
\Cnode*[fillstyle=solid,radius=0.04in](10,20){a}
\Cnode*[fillstyle=solid,radius=0.04in](10,40){b}
\Cnode[radius=0.04in](30,20){x}
\Cnode[radius=0.04in](30,40){y}
\ncline{-}{a}{b}
\ncline{-}{x}{y}
\ncline{-}{y}{a}
{\small
\rput[l](34,16){$2$}
\rput[l](34,44){$1$}}
\end{pspicture}
&\hspace{5pt}&
\psset{xunit=0.014in, yunit=0.014in}
\psset{linewidth=0.005in}
\begin{pspicture}(0,0)(40,60)
\Cnode*[fillstyle=solid,radius=0.04in](10,20){a}
\Cnode*[fillstyle=solid,radius=0.04in](10,40){b}
\Cnode[radius=0.04in](30,20){x}
\Cnode[radius=0.04in](30,40){y}
\ncline{-}{a}{b}
\ncline{-}{x}{y}
\ncline{-}{x}{b}
{\small
\rput[l](34,16){$1$}
\rput[l](34,44){$2$}}
\end{pspicture}
\\
\psset{xunit=0.014in, yunit=0.014in}
\psset{linewidth=0.005in}
\begin{pspicture}(0,0)(40,60)
\Cnode*[fillstyle=solid,radius=0.04in](10,30){a}
\Cnode*[fillstyle=solid,radius=0.04in](30,30){b}
\Cnode[radius=0.04in](20,50){x}
\Cnode[radius=0.04in](40,50){y}
\ncline{-}{a}{x}
\ncline{-}{x}{b}
\ncline{-}{y}{b}
{\small
\rput[l](44,54){$1$}
\rput[r](16,54){$2$}}
\end{pspicture}
&\hspace{5pt}&
\psset{xunit=0.014in, yunit=0.014in}
\psset{linewidth=0.005in}
\begin{pspicture}(0,0)(40,60)
\Cnode*[fillstyle=solid,radius=0.04in](10,30){a}
\Cnode*[fillstyle=solid,radius=0.04in](30,30){b}
\Cnode[radius=0.04in](20,10){x}
\Cnode[radius=0.04in](40,10){y}
\ncline{-}{a}{x}
\ncline{-}{x}{b}
\ncline{-}{y}{b}
{\small
\rput[l](44,6){$1$}
\rput[r](16,6){$2$}}
\end{pspicture}
&\hspace{5pt}&
\psset{xunit=0.014in, yunit=0.014in}
\psset{linewidth=0.005in}
\begin{pspicture}(0,0)(40,60)
\Cnode*[fillstyle=solid,radius=0.04in](10,30){a}
\Cnode*[fillstyle=solid,radius=0.04in](30,30){b}
\Cnode[radius=0.04in](20,10){x}
\Cnode[radius=0.04in](20,50){y}
\ncline{-}{a}{x}
\ncline{-}{x}{y}
{\small
\rput[l](24,6){$1$}
\rput[l](24,54){$2$}}
\end{pspicture}
&\hspace{5pt}&
\psset{xunit=0.014in, yunit=0.014in}
\psset{linewidth=0.005in}
\begin{pspicture}(0,0)(40,60)
\Cnode*[fillstyle=solid,radius=0.04in](10,30){a}
\Cnode*[fillstyle=solid,radius=0.04in](30,30){b}
\Cnode[radius=0.04in](20,10){x}
\Cnode[radius=0.04in](20,50){y}
\ncline{-}{x}{y}
\ncline{-}{y}{b}
{\small
\rput[r](16,6){$2$}
\rput[r](16,54){$1$}}
\end{pspicture}
\end{tabular}
\end{center}
\caption{The eight $2$-element extensions when $|P|=2$. The points labelled $1$ and $2$ correspond to the distinguished elements $\Ext_1(P^a)$ and $\Ext_2(P^a)$ respectively.}\label{fig-two-point-posets}
\end{figure}
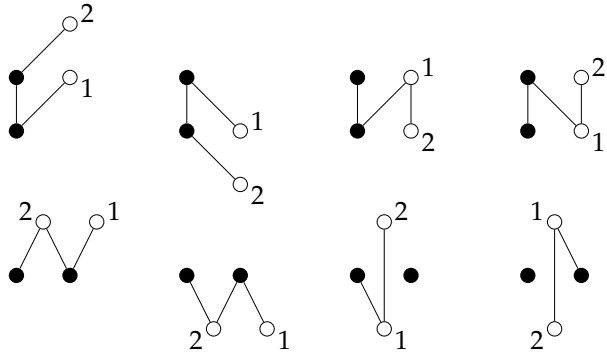

If every $A_s$ is a singleton, then $P$ is already simple, but for the purposes of the induction we must add two elements to form the four extensions $P^{\downarrow}$, $P^{\uparrow}$, $P^\updown$ and $P^\downup$. First observe that $P$ can contain neither a unique maximum nor a unique minimum element.  To form $P^\uparrow$ we find any two maxima of $P$, place $\Ext_1(P^\uparrow)$ above one of them, and place $\Ext_2(P^\uparrow)$ above both. We form $P^\downarrow$ analogously, and $P^{\updown}$ and $P^\downup$ are formed by adjoining two elements to any chosen element of $P$; see Figure~\ref{fig-simple-poset}. Note that all four extensions are necessarily simple.

\begin{figure}
\begin{center}
\begin{tabular}{ccccccc}
\psset{xunit=0.014in, yunit=0.014in}
\psset{linewidth=0.005in}
\begin{pspicture}(0,-10)(70,70)
\psccurve[linewidth=0.01in](65,15)(65,45)(15,45)(15,15)
\pnode(25,20){p1}
\pnode(35,20){p3}
\pnode(45,20){p2}
\pnode(55,20){p4}
\Cnode*[radius=0.04in](25,40){m1}
\Cnode*[radius=0.04in](55,40){m2}
\Cnode[radius=0.04in](25,60){q1}
\Cnode[radius=0.04in](55,60){q2}
\ncline[linestyle=dashed]{-}{p1}{m1}
\ncline[linestyle=dashed]{-}{p3}{m1}
\ncline[linestyle=dashed]{-}{p2}{m2}
\ncline[linestyle=dashed]{-}{p4}{m2}
\ncline{-}{m1}{q1}
\ncline{-}{m1}{q2}
\ncline{-}{m2}{q2}
{\small
\rput[r](21,64){$1$}
\rput[l](59,64){$2$}}
\end{pspicture}
&\hspace{10pt}&
\psset{xunit=0.014in, yunit=0.014in}
\psset{linewidth=0.005in}
\begin{pspicture}(0,-10)(70,70)
\psccurve[linewidth=0.01in](65,15)(65,45)(15,45)(15,15)
\pnode(25,40){p1}
\pnode(35,40){p3}
\pnode(45,40){p2}
\pnode(55,40){p4}
\Cnode*[radius=0.04in](25,20){m1}
\Cnode*[radius=0.04in](55,20){m2}
\Cnode[radius=0.04in](25,0){q1}
\Cnode[radius=0.04in](55,0){q2}
\ncline[linestyle=dashed]{-}{p1}{m1}
\ncline[linestyle=dashed]{-}{p3}{m1}
\ncline[linestyle=dashed]{-}{p2}{m2}
\ncline[linestyle=dashed]{-}{p4}{m2}
\ncline{-}{m1}{q1}
\ncline{-}{m1}{q2}
\ncline{-}{m2}{q2}
{\small
\rput[r](21,-4){$1$}
\rput[l](59,-4){$2$}}
\end{pspicture}
&\hspace{10pt}&
\psset{xunit=0.014in, yunit=0.014in}
\psset{linewidth=0.005in}
\begin{pspicture}(0,-10)(90,50)
\psccurve[linewidth=0.01in](65,15)(65,45)(15,45)(15,15)
\pnode(45,15){p1}
\pnode(45,45){p2}
\Cnode*[radius=0.04in](60,30){p}
\Cnode[radius=0.04in](80,50){q1}
\Cnode[radius=0.04in](80,30){q2}
\ncline[linestyle=dashed]{-}{p1}{p}
\ncline[linestyle=dashed]{-}{p2}{p}
\ncline{-}{p}{q1}
\ncline{-}{q1}{q2}
{\small
\rput[l](84,54){$1$}
\rput[l](84,26){$2$}}
\end{pspicture}
&\hspace{10pt}&
\psset{xunit=0.014in, yunit=0.014in}
\psset{linewidth=0.005in}
\begin{pspicture}(0,-10)(90,50)
\psccurve[linewidth=0.01in](65,15)(65,45)(15,45)(15,15)
\pnode(45,15){p1}
\pnode(45,45){p2}
\Cnode*[radius=0.04in](60,30){p}
\Cnode[radius=0.04in](80,10){q1}
\Cnode[radius=0.04in](80,30){q2}
\ncline[linestyle=dashed]{-}{p1}{p}
\ncline[linestyle=dashed]{-}{p2}{p}
\ncline{-}{p}{q1}
\ncline{-}{q1}{q2}
{\small
\rput[l](84,6){$1$}
\rput[l](84,34){$2$}}
\end{pspicture}\\
$P^\uparrow$&&$P^\downarrow$&&$P^\updown$&&$P^\downup$
\end{tabular}
\end{center}
\caption{The four $2$-element simple extension of an arbitrary simple poset. The points labelled $1$ and $2$ correspond to the distinguished elements $\Ext_1(P^a)$ and $\Ext_2(P^a)$, respectively.}\label{fig-simple-poset}
\end{figure}
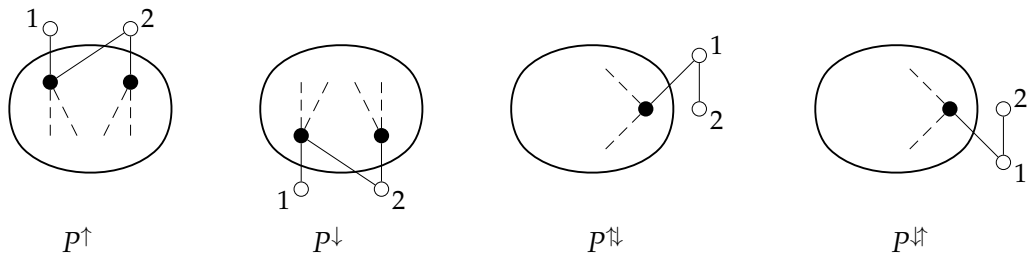

When at least one $A_s$ has more than one element, induction may be used on each such interval to form the extensions $A_s^{a_s}$
($a_s\in\{\uparrow,\downarrow,\updown,\downup\}$), which we utilise to construct $P^a$.  The basic idea is to consider the non-trivial blocks in turn, and for each of them determine which extension $A_s^{a_s}$ to use so that $A_s^{a_s}$ can share one distinguished point with the previous block, and one with the subsequent block.  The first and the last blocks will contribute one distinguished point each to become the distinguished points of $P^a$.  To be more precise, assume (without loss of generality) that the groundset of $S$ is $\{1,\dots,|S|\}$, and that the first $t$ of these elements correspond to the non-trivial blocks of $P$.  Define elements $a_s\in\{\uparrow,\downarrow,\updown,\downup\}$, $1\le s\le t$, as follows:
$$
a_s
=
\left\{\begin{array}{ll}
\left\{ \begin{array}{ll} \uparrow & \mbox{if } a\in\{ \uparrow,\updown \} \mbox{ and } 1 \ngtr_S  2\\
                              \downarrow & \mbox{if } a=\downarrow \mbox{ and } 1 >_S 2 \\
                              \updown & \mbox{if } a\in\{ \uparrow,\updown \} \mbox{ and } 1 >_S 2\\
                              \downup & \mbox{otherwise}
              \end{array} \right.
&\mbox{if $s=1$,}\\
\left\{ \begin{array}{ll} \uparrow & \mbox{if } s-1 \nless_S s \mbox{ and } s \ngtr_S  s+1\\
                              \downarrow & \mbox{if } s-1 <_S s \mbox{ and } s >_S s+1\\
                              \updown & \mbox{if } s-1\nless_S s \mbox{ and } s >_S s-1\\
                              \downup & \mbox{otherwise}
              \end{array} \right.
&\mbox{if $2\le s\le t-1$, or}\\
\left\{ \begin{array}{ll} \uparrow & \mbox{if } t-1\nless_S t \mbox{ and } a=\uparrow\\
                              \downarrow & \mbox{if } t-1 <_S t \mbox{ and } a\in\{\downarrow,\updown\}\\
                              \updown & \mbox{if } t-1\nless_S t \mbox{ and } a\in\{\downarrow,\updown\}\\
                              \downup & \mbox{otherwise}.
              \end{array} \right.
&\mbox{if $s=t$.}
\end{array}\right.
$$
For $1\le s\le t$, we define new, extended blocks $B_s=\Core(A_s^{a_s})$.  For the remaining elements of $S$, $t+1\le s\le |S|$, $A_s$ is the trivial poset on one element, and we set $B_s=A_s$.

Now we form the extension $S[B_s: s\in S]$ by adding $t+1$ linking points $\ell_s$ ($s=0,\ldots,t$), as follows:
$\ell_0$ is related to the elements of $B_1$ in exactly the same way as $\Ext_1(A_1^{a_1})$ is related to the elements of $B_1=\Core(A_1^{a_1})$;
$\ell_s$ ($1\leq s\leq t-1$) is related to the elements of $B_s$ in the same way as $\Ext_2(A_s^{a_s})$ is related to the elements of $B_s=\Core(A_s^{a_s})$, and $\ell_s$ is related to the elements of $B_{s+1}$ in the same way as $\Ext_1(A_{s+1}^{a_{s+1}})$ is related to the elements of $B_{s+1}=\Core(A_{s+1}^{a_{s+1}})$;
$\ell_t$ is related to the elements of $B_t$ in exactly the same way as $\Ext_2(A_t^{a_t})$ is related to the elements of $B_t=\Core(A_t^{a_t})$;
these comparisons and their consequences via transitivity are the only comparisons that the elements $\ell_s$ satisfy.
Finally, we stipulate $\Ext_1(P^a)=\ell_0$, $\Ext_2(P^a)=\ell_t$.

To see that one such $P^a$ ($a\in\{\uparrow,\downarrow,\updown,\downup\}$) is simple, a case analysis needs to be carried out, but as this is similar to the permutation case, we omit the details.  The only intervals that can arise in any such $P^a$ must contain $\ell_t$ and be contained in $B_t\cup\{\ell_t\}$, and these intervals are permitted by the inductive hypothesis for all but one of the extensions. To see that there is indeed one extension with no such interval, however, note that by the simplicity of $S$, $B_t$ must be related to some other block $B_s$ of $P$, and by symmetry suppose $B_s<B_t$. Taking $\ell_t$ to be a new minimum (i.e.\ the extension $P^a$ that we are considering satisfies $a\in\{\downarrow,\updown\}$), we note that any interval $I$ in $B_t\cup\{\ell_t\}$ cannot be an interval in $P^a$, since there is some $x\in I\cap B_t$ with $x>B_s$, while $\ell_t\ngtr B_s$.

In the degenerate case we have $P=S[A_1,A_2]$, where $S$ is the 2-element chain (so $A_1<A_2$) or the 2-element
antichain (so that $A_1$ and $A_2$ are incomparable). Supposing first that $S$ is a chain, if possible pick $A_1$ so that it has no unique maximum. Assuming this is possible and that the resulting $A_2$ is not a singleton, identify the distinguished points of the extensions $A_1^{a_1}$ and $A_2^{a_2}$ in the same way as the non-degenerate case, with $\Ext_1(P^a)=\Ext_1(A_1^{a_1})$ and $\Ext_2(P^a)=\Ext_2(A_2^{a_2})$. On the other hand, if the resulting $A_2$ is a singleton, we form $S[\Core(A_1^a),A_2]$ and then add the distinguished points to $A_1^a$ to serve as the distinguished points of $P^a$. (Note that by the requirement that $A_1^\downup$ contained a maximal element that is not contained in any proper interval and that $\Ext_2(P^\downup)$ was contained in every proper interval, there are no proper intervals in $P^\downup$.)

If we cannot pick $A_1$ to have no unique maximum but we can instead pick $A_2$ to have no unique minimum, then we may argue by symmetry. This leaves only the case where $P$ is a chain, which is handled by appealing directly to the permutation case; the chain $P$ corresponds to an increasing permutation $\pi=$ under the following mapping:  $i< j$ in $P$ if and only if both $i<j$ and $\pi(i)<\pi(j)$. Under this mapping the simple extensions of increasing permutations map to simple extensions of linear orders, and by appealing to the inductive construction of Theorem~\ref{thm-permutations} this allows us to form $P^\downarrow$ and $P^\downup$, one of which is simple. We form $P^\uparrow$ and $P^\updown$ by using an analogous mapping that connects chains to decreasing permutations.

This leaves the case where $S$ is a 2-element antichain. By analogy with the case where $S$ is a two-element chain, first attempt to pick $A_1$ so that it is not itself decomposable by an antichain and, if possible, so that $A_1$ is not a singleton. If this can be done leaving $A_2$ with $|A_2|\geq 2$, then we proceed by using the same construction as the nondegenerate case, while if $A_2$ is a singleton then we form $P^\downup$ and $P^\updown$ simply by taking the corresponding extensions of $A_1$ and leaving the single element of $A_2$ unrelated to any other elements. We ensure that $P^\uparrow$ and $P^\downarrow$ are simple extensions of $P$ by adding the relation $\Ext_2(A_1)>A_2$ in $P^\uparrow$ and $\Ext_2(A_1)<A_2$ in $P^\downarrow$. This leaves only the case where $P$ is an antichain, which is covered by appealing to the connection between posets and permutations.\end{proof}

As we may expect by its similarity to the permutation case, precisely $\lceil(n+1)/2\rceil$ additional elements are required when $(P,<)$ is a linear order.  Note further that this bound is tight, because of its connection with permutations.

On the other hand, when $(P,<)$ is an antichain, a simple extension does not require this many points by the following connection with graphs: the \emph{comparability graph} $G(P,<)$ of a poset $(P,<)$ is a graph with vertex set $P$, and edges $p\sim q$ if and only if either $p<q$ or $q<p$.%
\footnote{The converse operation --- forming a poset from a graph --- is called a \emph{transitive orientation} of (the edges of) $G$. This connection between posets and graphs arises in a number of combinatorial problems --- see M\"ohring~\cite{mohring:algorithmic-asp:} for a survey.} %
It is then easy to appeal to Theorem~\ref{thm-graphs} to create simple extensions of antichains that require just $\lceil\log_2 (n+1)\rceil$ additional points, and again this bound is tight. In fact, if the simple extension of a graph admits a transitive orientation then this extension could be interpreted as a simple extension of a poset.


\section{Digraphs and Oriented Graphs}\label{sec-digraphs}
An arbitrary digraph corresponds to the most general type of binary relational structure, i.e.\ the structure $\A$ contains just one binary relation $R$ for which there are no restrictions. Elements of the relation may be viewed as directed edges on a graph, so that two vertices $u$ and $v$ of a digraph may be related in one of four different ways: $u\rightarrow v$, $u\leftarrow v$, $u\leftrightarrow v$, or $u$ and $v$ have no edge between them.

We obtain a general bound on the number of additional vertices required by analogy with the graph case: where we had two possible ``types'' of connection (i.e. an edge or no edge) between an old and a new vertex for graphs, we now have four. The result of this is that a new vertex in a digraph can do the work of two new vertices in a graph, so half as many vertices are required.

\begin{theorem}
A digraph $D$ has a simple extension with at most $m=\lceil\log_4(|V(D)|+1)\rceil$ additional vertices.
\end{theorem}

\begin{proof}
The \emph{underlying graph} $G_D$ of a digraph $D$ is the graph formed by replacing each directed edge with an undirected one, i.e.\ $uv\in E(G)$ if and only if $u\rightarrow v$ or $v\rightarrow u$. Moreover, it is easy to see that a set of vertices that forms an interval in $D$ will also form an interval in $G_D$. By Theorem~\ref{thm-graphs}, we can construct a simple extension $G_D^*$ of $G_D$ by adding a set $B=\{b_1,b_2,\ldots,b_m\}$ of $m\leq\lceil\log_2(|V(G_D)|+1)\rceil$ vertices. Now, noting that $2\lceil\log_4(|V(D)|+1)\rceil \geq \lceil\log_2(|V(G_D)|)+1\rceil$, we add a set $C=\{c_1,\ldots,c_\ell\}$ of $\ell= \lceil m/2\rceil$ additional vertices to $D_G$, satisfying, for every $v \in V(D)$ and $i=1,2,\ldots, \ell$,
\begin{eqnarray*}
v\rightarrow c_i&\textrm{if and only if}& vb_{2i-1}\in E(G_D^*)\textrm{, and,}\\
c_i\rightarrow v&\textrm{if and only if}& vb_{2i}\in E(G_D^*)\textrm{ and }2i\leq \ell.\end{eqnarray*}
Thus every vertex $c_i$ of $C$ has outneighbourhood matching the neighbourhood of $b_{2i-1}$, and inneighbourhood matching $b_{2i}$. It is now straightforward to check that this is a simple extension.
\end{proof}

This bound can be seen to be tight for a complete digraph --- a digraph for which $u\leftrightarrow v$ for every pair of distinct vertices $u$ and $v$ --- and an empty digraph, again by analogy with the graph case.

However, if our digraph forms a linear order (and is hence transitive), our structure forms a transitive tournament. Appealing to Theorem~\ref{tournament}, only one additional point is required to form a simple tournament unless the tournament had an odd number of vertices, when two are required. Here, however, we do not require the resulting simple extension to be a tournament, and this added freedom allows us to form a simple one-point extension in every case. See Figure~\ref{fig-linear-digraph} for an example.

Extending the connection with tournaments, any digraph that is also a tournament must consequently have a simple one-vertex extension. A natural question then arises: which digraphs have simple extensions requiring only one additional vertex? Clearly structures which are ``close'' to tournaments will have this property, and again the answer must lie in the substitution decomposition tree, and in particular how we decompose a structure with a quotient of size $2$.

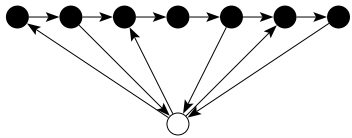
\begin{figure}
\begin{center}
\psset{xunit=0.014in, yunit=0.014in}
\psset{linewidth=0.005in}
\begin{pspicture}(0,0)(120,60)
\Cnode*[fillstyle=solid,radius=0.06in](0,50){p1}
\Cnode*[fillstyle=solid,radius=0.06in](20,50){p2}
\Cnode*[fillstyle=solid,radius=0.06in](40,50){p3}
\Cnode*[fillstyle=solid,radius=0.06in](60,50){p4}
\Cnode*[fillstyle=solid,radius=0.06in](80,50){p5}
\Cnode*[fillstyle=solid,radius=0.06in](100,50){p6}
\Cnode*[fillstyle=solid,radius=0.06in](120,50){p7}
\Cnode[radius=0.06in](60,10){x}
\psset{arrowsize=3pt 2}
\psset{arrowinset=0.2}
\ncline{->}{p1}{p2}
\ncline{->}{p2}{p3}
\ncline{->}{p3}{p4}
\ncline{->}{p4}{p5}
\ncline{->}{p5}{p6}
\ncline{->}{p6}{p7}
\ncline{<-}{p1}{x}
\ncline{->}{p2}{x}
\ncline{<-}{p3}{x}
\ncline{->}{p5}{x}
\ncline{<-}{p6}{x}
\ncline{->}{p7}{x}
\end{pspicture}
\end{center}
\caption{A one-vertex simple extension of a linear digraph on $7$ vertices.}\label{fig-linear-digraph}
\end{figure}

\paragraph{Oriented Graphs}
An {\it oriented graph} $G$ is a graph in which each edge has a specified direction, and it can be described as a relational structure $\mathcal{A}_G$ on the language $\mathcal{L}=\{\rightarrow, n_\rightarrow=2\}$ where $\rightarrow$ is an asymmetric binary relation, i.e.\ for each $u,v\in\dom(\mathcal{A}_G)=V(G)$, $u\rightarrow v$ implies $v\not\rightarrow u$. The difference between these and digraphs is that oriented graphs permit only three types of connection between two vertices, and so in particular there is no oriented graph corresponding to a complete graph. There is, however, an oriented graph corresponding to the empty graph, and it is this that gives rise to the upper bound.

\begin{theorem}
An oriented graph $G$ has a simple extension with at most $m=\lceil\log_3(|V(G)|+1)\rceil$ additional vertices.
\end{theorem}

When the oriented graph forms a tournament, it is easy to check that only one additional vertex is required --- note in particular that the simple extension in Figure~\ref{fig-linear-digraph} is also an oriented graph.


\section{Relations with Higher Arity}\label{sec-higher-arity}

When a relational structure is defined on a language containing a $k$-ary relation for some $k\geq 3$, simple extensions are often trivial to find and depend only on the type of relation. We begin with a straightforward observation.

\begin{lemma}\label{lem-one-relation-will-do}
A relational structure $\A$ on the language $\L=\{ R_1,R_2,\ldots, n_{R_1}, n_{R_2},\ldots \}$ is simple if there is some $k$ for which the structure $\A_{R_k}$, defined on the ground set $\dom(\A_{R_k})=\dom(\A)$ and language $\L_{R_k}=\{R_k,n_{R_k}\}$, is simple. \end{lemma}

\begin{proof}
Suppose, for a contradiction, that $I$ is a nontrivial interval of $\A$, and in particular we may choose two distinct elements $x,y\in I$. Since $\A_{R_k}$ is simple, there exists $z\in \dom(\A_{R_k})\setminus I=\dom(\A)\setminus I$ for which $x$ and $y$ are not interchangeable in any relationship of $R_k$ containing $z$ and at least one of $x$ and $y$. This, however, is also true in $\A$.\end{proof}

Subsequently, if a relational structure $\mathcal{A}$ contains an arbitrary $k$-ary relation $R$ with $k\geq 3$, only one additional element is required, regardless of what other relations the structure may hold:

\begin{theorem}\label{thm-arbitrary}
A relational structure $\A$ containing an arbitrary $n_R$-ary relation $R$ with $n_R\geq 3$ has a one-element simple extension.
\end{theorem}

\begin{proof}By Lemma~\ref{lem-one-relation-will-do} we need only prove that we may form a one-element simple extension for a structure $\A$ defined on the language $\L=\{R,n_R\}$ consisting of only one relation. Furthermore, we assume that $n_R=3$, noting that the result is easily extended if our chosen arbitrary relation has higher arity.

Choose any simple binary relational structure $\B$ on the set $\dom(\A)$ with language $\{B, n_B=2\}$ --- such a structure exists for every size of set by taking, for example, a path in the graph-theoretic sense unless $|\A|\leq 3$, in which case we may choose a simple tournament. We now add an additional vertex $x$ to $\A$ and define relations between $x$ and $\dom(\A)$ to form $\A^*$ as follows: for every $(u,v)\in B^{\B}$, put $(x,u,v)\in R^{\A^*}$. (Note that if $n_R>3$ we set $(x,\ldots,x,u,v)\in R^{\A^*}$.) We claim that $\A^*$ is simple. By the simplicity of $\B$, any interval containing more than one element of $\dom(\A)$ must contain all of $\dom(\A)$, so the only intervals we need to rule out are $I=\{w,x\}$ for some $w\in\dom(\A)$, and $I=\dom(\A)$.

For the first case, by the construction there exists at least one relationship of the form $(x,w,v)$ or $(x,v,w)$ in $R^{\A^*}$ for some $v\in\dom(\A)$ with $v\neq w$, and without loss we may assume the former. Since $I$ is an interval, this implies that $(x,x,v)\in R^{\A^*}$, a contradiction as there are no relationships of the form $(x,x,\cdot)$ in $R^{\A^*}$. In the second case, $I=\dom(\A)$ implies that $(x,u,v)\in R^{\A^*}$ for every $u,v\in\dom{\A}$, which implies that $\B$ can be thought of as a complete graph, a contradiction since $\B$ was chosen explicitly to be simple.
\end{proof}

Structures which have more restrictive relations do also sometimes have one point extensions. For example, a suitably large relational structure containing a $k$-ary relation for some $k\geq 3$ that is irreflexive in its entries is easily found to have such an extension:

\begin{theorem}For $k\geq 3$, a relational structure with $n\geq k$ elements on a $k$-ary irreflexive relation has a one vertex simple extension.\end{theorem}

\begin{proof}
Fix $k\geq 3$, and let $R$ be a $k$-ary irreflexive relation. By Lemma~\ref{lem-one-relation-will-do} it suffices to form a one-element simple extension $\A^*$ of a relational structure $\A$ on the language $\L=\{R,k\}$ whose ground set consists of $n\geq k$ vertices. Set $\dom{\A^*}=\dom{\A}\cup\{x\}$, and define the set of relations $R^{\A^*}$ as follows: $(x_1,\ldots,x_k)\in R^{\A^*}$ if and only if either $(x_1,\ldots,x_k)\in R^\A$ or $x=x_i$ for exactly one $i$ and $(x_1,\ldots,x_{i-1},x_{i+1},\ldots,x_k)$ is a $(k-1)$-tuple of distinct elements from $\dom(\A)$.

It follows routinely that $\A^*$ is simple: suppose that $I$ is a proper interval of $\A^*$ containing $u,v\in \dom(\A^*)$. By the construction there is at least one relation in $R^{\A^*}$ of the form $(u,v,x_3,\ldots,x_k)$ where $x_3\not\in I$, but this then means that $(u,u,x_3,\ldots,x_k)\in R^{\A^*}$, a contradiction since $R$ is irreflexive.
\end{proof}

Note in particular that the above proof immediately answers the problem for hypergraphs since the construction ensures that if $\A$ is a hypergraph then so is $\A^*$ (i.e.\ $R^{\A^*}$ is symmetric in its entries).

\section{Concluding Remarks}\label{sec-conclusion}

\paragraph{Tight bounds.} Although the bounds in this paper are tight in the sense that there are certain structures which require as many points as the bound permits, on a typical example our constructions do quite a bit better. As noted in Section~\ref{sec-graphs}, more than two vertices are required in the graph case only if there is a large degenerate node in the substitution decomposition tree, all of whose children are leaves. Similar analysis of the substitution decomposition tree may be used in the cases of permutations, oriented graphs and digraphs to determine exactly how many points will be required in any particular case. More complicated is the poset case: we cannot simply appeal to the permutation case because degenerate decompositions whose quotients are antichains can be treated as if they are empty graphs, which require only $O(\log_2 n)$ additional vertices.

\paragraph{Average number of additional points required.} In the case of graphs and any relational structure defined on a single asymmetric relation (e.g. tournaments, posets and oriented graphs), almost all structures are already simple (see M\"ohring~\cite{mohring:on-the-distribu:}) and hence no points are required. The permutation case is more interesting, because here one expects to find two intervals of size two (see, for example,~\cite{brignall:a-survey-of-sim:}).

\paragraph{Simple extensions of infinite structures.} Proposition~\ref{tournament-one-point} as proved by Erd{\H{o}}s, Hajnal and Milner applies to tournaments of arbitrary cardinality, and it is natural to wonder what (if anything) can be said for other structures. Note that for many such structures it is easy to construct examples that require infinitely many additional points (e.g., the countably infinite complete graph). On the other hand, Theorem~\ref{thm-arbitrary} holds so long as the structure contains an arbitrary $n_R$-ary relation with $n_R\geq 3$.

\bibliographystyle{acm}
\bibliography{../refs}

\end{document}